\documentclass[12pt,reqno]{amsart}
\usepackage[left=1in,top=1in,right=1in,bottom=1in,letterpaper]{geometry}

\usepackage[utf8]{inputenc}
\usepackage{amsmath}
\usepackage{amsthm}
\usepackage{amssymb}
\usepackage{amsfonts}
\usepackage{amsaddr}
\usepackage{tabu}
\usepackage{float}
\usepackage{graphicx}
\usepackage{accents}
\usepackage{listings}
\usepackage{epstopdf}
\usepackage{color}
\usepackage{caption}
\usepackage{subcaption}
\usepackage{hyperref}
\usepackage{cancel}
\usepackage{multicol}
\usepackage{mathtools}
\usepackage{pgfplots}
\pgfplotsset{compat=newest}
\usepackage{tikz}
\usetikzlibrary{decorations.markings}
\usetikzlibrary{snakes}

\usepackage{algorithm,algorithmic,mathtools}
\usepackage{varwidth}

\DeclareCaptionType[placement={!ht}]{listing}[Listing][Code Listings]

\usetikzlibrary{shapes.misc}

\tikzset{cross/.style={cross out, draw=black, minimum size=2*(#1-\pgflinewidth), inner sep=0pt, outer sep=0pt},
	cross/.default={1pt}}

\DeclareMathOperator*{\esssup}{ess\,sup}

\newtheorem{thm}{Theorem}[section]

\newtheorem{lem}[thm]{Lemma}
\newtheorem{prop}[thm]{Proposition}
\newtheorem{prob}[thm]{Problem}
\theoremstyle{definition}
\newtheorem{defn}[thm]{Definition}
\theoremstyle{remark}
\newtheorem{rem}[thm]{Remark}

\numberwithin{equation}{section}

\begin{document}
	
\title[Finite dimensional backstepping controller design]{Finite dimensional backstepping controller design}%

\author{Varga K. Kalantarov$^{\MakeLowercase{a}}$, Türker Özsarı$^{\MakeLowercase{b},*}$ and Kemal Cem Yılmaz$^{\MakeLowercase{c}}$}
\address{$^a$Department of Mathematics, Koç University\\ Sarıyer, İstanbul, 34450 Turkey}
\address{$^b$Department of Mathematics, Bilkent University\\ Çankaya, Ankara, 06800 Turkey}
\address{$^c$Department of Mathematics, İzmir Institute of Technology\\ Urla, Izmir, 35430 Turkey}
\thanks{This research was funded by TÜBİTAK 1001 Grant \#122F084.}
\thanks{T. Özsarı's research was partially supported by Science Academy's Young Scientist Award (BAGEP 2020)}
\thanks{*Corresponding author: Türker Özsarı, turker.ozsari@bilkent.edu.tr}
\keywords{backstepping, boundary feedback, reaction-diffusion equation, stabilization.}
\subjclass[2020]{35B40, 35K57, 93C20, 93D15, 93D20, 93D23}

\begin{abstract}
We introduce a finite dimensional version of backstepping controller design for stabilizing solutions of PDEs from boundary. Our controller uses only a finite number of Fourier modes of the state of solution, as opposed to the classical backstepping controller which uses all (infinitely many) modes.   We apply our method to the reaction-diffusion equation, which serves only as a canonical example but the method is applicable also to other PDEs whose solutions can be decomposed into a slow finite-dimensional part and a fast tail, where the former dominates the evolution in large time. One of the main goals is to estimate the sufficient number of modes needed to stabilize the plant at a prescribed rate.  In addition, we find the minimal number of modes that guarantee the stabilization at a certain (unprescribed) decay rate. Theoretical findings are supported with numerical solutions.
\end{abstract}
\maketitle
\section{Introduction}
\label{sec:introduction}

The goal of this paper is to introduce a finite-dimensional analogue of the classical backstepping control algorithm for stabilizing solutions of partial differential equations (PDEs) from boundary.  As a canonical example, we will use the reaction diffusion equation posed on an interval with Dirichlet boundary conditions. However, the method is applicable also to other evolutionary equations whose solutions decompose into a slow finite-dimensional part and a fast tail, where the former dominates the dynamics in long time. To this end, let us consider the following plant:
\begin{eqnarray} \label{pde_nl}
	\begin{cases}
		u_t - \nu u_{xx} - \alpha u + \kappa u^3 = 0, \, x\in(0,L), t > 0, \\
		u|_{x=0} = 0, \quad u|_{x=L} = g, \\
		u|_{t=0} = u_0.
	\end{cases}
\end{eqnarray}
Here, $\nu, \alpha > 0$ are given constant coefficients and $g$ is a soughtafter feedback controller that involves only finitely many Fourier sine modes of $u$. We take either $\kappa = 1$ or $\kappa = -1$, where the former one may lead to a uniformly bounded nontrivial stationary solution in-time whereas the latter one may lead the solutions to blow-up in a finite time, if $g = 0$. More precisely, in the absence of control and for certain values of $\nu$, $\alpha$ and $L$ together with $\kappa = 1$, zero equilibrium solution may be either asymptotically stable or unstable: If $\nu\lambda_1 - \alpha > 0$, where $\lambda_1$ is the first eigenvalue of the operator $-\frac{d^2}{dx^2}$ subject to Dirichlet boundary conditions, \eqref{pde_nl} has a unique equilibrium solution $u \equiv 0$. For this case, it is asymptotically stable. Conversely, if $\nu\lambda_1 - \alpha < 0$, then there will be at least two nontrivial stationary solutions, exactly two of which are asymptotically stable, and all solutions bifurcate from the zero equilibrium. This leads to the fact that $u \equiv 0$ is no longer stable. Regarding the latter case where $\kappa = -1$, for some initial data, there exists a finite time $t_0$ such that the energy functional $E(t) = \int_0^L \left(\nu u_x^2 - u^4/2 - \alpha u^2\right)(x,t) dx \to \infty$ as $t \to t_0^-$ (e.g., see \cite{chafee,henry,sattinger} and \cite{zheng2004nonlinear} for a more detailed discussion on these topics). Thus, for each case, our aim is to construct a feedback law of the form \begin{equation}
	\label{feedbackcontroller}
	u(L,t) = \int_0^L \xi(y) \Gamma[P_N u](y,t)dy,
\end{equation}
such that all solutions are steered asymptotically to zero. Here, $\Gamma$ is a linear bounded operator on a certain $L^2-$based functional space, $\xi$ is a suitable smooth function to be constructed, and $P_N$ is the projection operator $P_N\varphi(x) = \sum_{j=1}^N e_j(x) \left(e_j(\cdot), \varphi(\cdot)\right)_2, \quad e_j(x) = \sqrt{\frac{2}{L}} \sin \left(\frac{j \pi x}{L}\right)$. We will refer to boundary feedbacks of above form as a finite-dimensional backstepping controller. The feedback controller $g$ in \eqref{pde_nl} will be defined to be the integral at the right hand side of \eqref{feedbackcontroller} and we explicitly denote this by $g(P_Nu(t))$.

\subsection{Preliminaries}\label{prelim}
For clarity of the exposition, here we provide some notation and inequalities used throughout.  
$L^p(0,L)$, $1 \leq p \leq \infty$ is the usual Lebesgue space and given $\varphi \in L^p(0,L)$ we denote its $L^p-$norm by $\|\varphi\|_{L^p(0,L)}$. For $p = 2$, we write $\|\varphi\|$ instead of $\|\varphi\|_{L^2(0,L)}$. Given $\ell > 0$, we represent the $L^2-$based Sobolev space by $H^\ell(0,L)$. $H_0^1(0,L)$ is the closed subspace of $H^1(0,L)$ that involves those functions in $H^1(0,L)$ which vanish (have zero trace) on the boundary of $(0,L)$. If $A:H^\ell(0,L) \to H^\ell(0,L)$ is a linear bounded operator, $\ell = 0,1$, we denote its operator norm as $\|A\|_{H^\ell(0,L) \to H^\ell(0,L)}$. We write $a \lesssim b$ to denote an inequality $a \leq c b$, where $c > 0$ may depend on only some fixed parameters. With $\lambda_{j}$, $j = 1, 2, \dotsc$, where $\lambda_j = j^2 \lambda_1=\left(\frac{j\pi }{L}\right)^2$, we denote the $j-$th eigenvalue of the operator $-\frac{d^2}{dx^2}$ subject to Dirichlet boundary conditions, and $e_j=\sqrt{\frac{2}{L}}\sin(\frac{j\pi x}{L})$ is the corresponding $L^2-$normalized eigenfunction. We introduce the (solution) spaces $$X_T^\ell\equiv L^\infty(0,T;H^\ell(0,L)) \cap L^2(0,T;H^{1+\ell}(0,L)), \ell=0,1.$$  
We equip these spaces with the norms
\begin{equation*}
	\|\psi\|_{X_T^\ell}^2
	\equiv\esssup_{t\in [0,T]}\|\psi(\cdot,t)\|_{H^{\ell}(0,L)}^2+\int_0^T\|\psi(\cdot,t)\|_{H^{1+\ell}(0,L)}^2dt.
\end{equation*}
The following inequalities are useful while carrying out some of the norm estimates.
\begin{itemize}
	\item{Cauchy's inequality with} $\epsilon$: $ab \leq \frac{\epsilon a^2}{2} + \frac{b^2}{2\epsilon}$, $\epsilon>0$.
	\item {Cauchy-Schwarz inequality:} For $\varphi,\psi \in L^2(0,L)$, $\left| \left(\varphi,\psi\right)_2 \right| \leq \|\varphi\| \|\psi\|.$
	\item{Poincaré inequality:} For $\varphi \in H_0^1(0,L)$ or for $\varphi\in H^1(0,L)$ with $\int_{0}^L\varphi(x) dx=0$, it follows $\|\varphi\|^2 \leq \lambda_1^{-1} \|\varphi^\prime\|^2.$ In particular, $\|\varphi^\prime\|_{2}^2\le \lambda_1^{-1}\|\varphi''\|_2^2$ for $\varphi\in H_0^1(0,L)\cap H^2(0,L)$. 
	\item{Poincaré type inequality:} For $u \in H_0^1(0,L)$, 
	\begin{equation*}
		\|u - P_N u\|^2 \leq \lambda_{N+1}^{-1} \|u^\prime\|^2, \quad \lambda_{N+1} = (N+1)^2 \lambda_1.
	\end{equation*}
	The proof of this inequality is easy, indeed with $\hat{u}_j=(u,e_j)_2$, one has
	\begin{equation*}
		\|u^\prime\|^2=\sum\limits_{j=1}^\infty\lambda_j\hat{u}_j^2\ge \sum\limits_{j=N+1}^\infty\lambda_j\hat{u}_j^2 \ge \lambda_{N+1}\sum\limits_{j=N+1}^\infty\hat{u}_j^2= \lambda_{N+1}\|u-P_Nu\|^2.
	\end{equation*}
	\item{Gagliardo-Nirenberg's inequality:} Let $u \in H_0^1(0,L)$, $p \geq 2$, $\alpha = \frac{1}{2} - \frac{1}{p}$. Then, $\|u\|_{L^p(0,L)} \leq c^* \|u^{\prime}\|^\alpha \|u\|^{1-\alpha},$ where $c^* > 0$  is the best GN-constant that depends on $p, \alpha$, and $L$.
\end{itemize}

\subsection{Problem statements}
We are interested in two different stabilization problems.
\begin{prob} \label{prob}
	Let $\nu, \alpha >0$ be such that $\nu\lambda_1 - \alpha < 0$ and $u_0 \in H^\ell(0,L)$, where $\ell=0$ or $\ell=1$. For \textit{given} $\gamma > 0$, can we find $N > 0$ and construct a feedback control law of the form $u(L,t) = g(P_Nu(t))$ such that the solution of \eqref{pde_nl} satisfies $\|u(t)\|_{H^\ell(0,L)} = \mathcal{O}(e^{-\gamma t})$ for $t \ge 0$?
\end{prob}

In contrast with the above \emph{rapid stabilization} problem, if one is only interested in decay at a certain unprescribed rate, then the problem takes the following form:

\begin{prob} \label{prob2}
	Let $\nu, \alpha >0$ be such that $\nu\lambda_1 - \alpha < 0$ and $u_0 \in H^\ell(0,L)$, where $\ell = 0$ or $\ell=1$.  What is the minimum value of $N$ for which the solution of \eqref{pde_nl} satisfies $\|u(t)\|_{H^\ell(0,L)} = \mathcal{O}(e^{-\gamma t}), t \ge 0$ for \textit{some} $\gamma > 0$ with a boundary feedback of the form $u(L,t) = g(P_Nu(t))$?
\end{prob}

\begin{rem}
	In Problem \ref{prob} and  Problem \ref{prob2}, we are not only interested in establishing the existence of a sufficiently large $N$ that fulfills the goal, in addition, we want to provide a precise estimate for $N$ in terms of problem parameters such as the decay rate and the coefficients in the main equation.
\end{rem}
Problem \ref{prob} and Problem \ref{prob2} will be treated by combining ideas from the theory of asymptotic dynamics of dissipative systems and  boundary control of PDEs via backstepping method, e.g., see \cite{krsticbook}. We give an affirmative answer to the question stated in Problem \ref{prob} in Theorems \ref{wpstab_lin} and \ref{wpstab_nonlin}. Regarding Problem \ref{prob2}, we find that $N$ is exactly equal to the instability level of \eqref{pde_nl}. More precisely, if the number of positive eigenvalues of the differential operator $-\nu \frac{d^2}{dx^2} - \alpha I$ with domain $\{\varphi \in H^2(0,L) \, | \, \varphi(0) = \varphi(L) = 0\}$ is $N$, then it suffices to design a boundary controller involving only the first $N$ Fourier sine modes of $u$. Corresponding results are stated in Theorem \ref{wpstab_lin2} and Theorem \ref{wpstab_nonlin2}.

\subsection{Discussion of previous work and motivation} \label{motivation}
Control and stabilization of evolutionary PDEs is an important topic and has become an area of interest for many researchers in recent decades. Several different approaches exist for designing control systems. In some of these systems, control inputs act from the interior of the medium, perhaps only from a local part. Other systems may use boundary controllers, especially when access to medium is restricted. The control mechanism is of feedback type if the controller depends on the state of the model. A backstepping-based boundary feedback controller is a classical example. See \cite{AaSmKr,KrGuSm,Liu03,SmCeKr,SmKr1,SmKr2} for its application to second-order PDEs and \cite{BaOzYi,CeCo,CoLu,OzAr,OzBa,OzYi} for its application to third-order PDEs.

Stabilization of a nonlinear diffusion equation (with a quadratic nonlinearity) via the standard backstepping transform was previously achieved by \cite{Yu14}. Designing a finite-dimensional controller for a similar problem as in the current paper is more challenging. For instance, the usual methods for proving bounded invertibility of the standard backstepping transformation fails in the current finite-dimensional case and one needs to give a new proof of such result. Moreover, proofs of decay rate estimates through multipliers are more delicate in the present case as projections are involved both in equations and in the estimates.  Designing a control system where the controller involves only finitely many parameters of the solution rather than the full state of the solution also offers an important practical use. 

The idea of stabilization of solutions of nonlinear parabolic equations by finite-dimensional controllers goes back to the pioneering works on 2D Navier-Stokes equations (see \cite{BaTi}, \cite{Fur1}).  The results obtained in these works and the pioneering results on infinite-dimensional disspative systems that generate 2D Navier-Stokes equations, nonlinear parabolic equations, nonlinear damped wave equations and related systems of PDEs (e.g., see \cite{BaVi}, \cite{Lad1}, \cite{FTR}) inspired further investigations on stabilization of these equations by internal and boundary finite-dimensional controllers (see \cite{AzTi}, \cite{Che}, \cite{KaTi}, \cite{KaOz}, \cite{LHM}, \cite{Mun1} and references therein). A number of papers are devoted to the boundary stabilization of nonlinear parabolic equations by finite-dimensional controllers (\cite{Bar1}, \cite{Mun1}, \cite{Mun2}). In particular, Barbu \cite{Bar1} studied the problem of stabilization of the stationary solution of the system
\begin{equation*}
	\begin{cases}
		u_t=\Delta u+f(x,u),\ \  \mbox{in} \ \ (0,\infty)\times \Omega,\\
		u=v, \ \ \mbox{on} \ \ (0,\infty)\times \Gamma_1, \ \ \frac{\partial u}{\partial n}=0  \ \ \mbox{on} \ \ (0,\infty)\times \Gamma_2, \\
		u(x,0)=u_0(x)\ \  \mbox{in}
		\ \ \Omega,
	\end{cases}
\end{equation*}
with feedback boundary controller depending on finitely many parameters. Here, $\Omega\subset \mathbb R^d$ is a bounded domain with the boundary $\partial \Omega =\Gamma_1\cup \Gamma_2$, where $\Gamma_1,\Gamma_2$ are disjoint connected components of $\partial \Omega.$ They assume that the system, $\{(\partial e_j/\partial n) \, | \, 1 \leq j \leq N\}$, of unstable Fourier modes is linearly independent on $\Gamma_1$. Their strategy works successfuly for $d \geq 2$. However, in one spatial dimension, this assumption is true only if the number of unstable Fourier modes is one. This situation results in a restriction $\alpha < \lambda_2$ on the coefficient of the zero-order term of the main equation, where $\lambda_2$ is the second eigenvalue of the differential operator $-\frac{d^2}{dx^2}$ with domain $\{\varphi \in H^2(0,L) \, | \, \varphi^\prime(0) = \varphi(L)  = 0\}$. The strategy that we design for Problems \ref{prob} and Problem \ref{prob2} will work for arbitrary choice of $\alpha$, i.e., arbitrary levels of instability.

Recently, stabilization of the 1D linear heat equation by observer-based finite-dimensional controllers using in-domain point measurement and boundary measurement was studied in \cite{KaFr} and \cite{LhPr}, respectively. There are also similar results for the semilinear heat equation, see \cite{Kat23}. Their strategy relies on homogenizing the boundary conditions by changing the variables to transfer the boundary control input into the domain, and then decomposing the main equation with respect to the Fourier modes. Utilizing the fact that the eigenvalues $\lambda_j, j \in \mathbb{Z}^+$, of the differential operator are countable and ordered, and therefore unstable Fourier modes are finitely many, say the first $N_0$, they control only first $N_0$ modes by considering $N_0-$dimensional system of ordinary differential equations (ODEs). Another recent study on the stabilization of linear heat equation in higher dimensions is \cite{Feng2} (see also \cite{Feng1} for 1D case), which is again based on stabilizing the model via spectral truncation stabilizer.

These studies mentioned in the previous paragraph are based on stabilizing the infinite-dimensional dynamical systems by controlling a corresponding finite-dimensional system, which  eventually amounts to use finite-dimensional feedback controllers. In our present study, we will follow a different approach that allows us to stabilize the zero equilibrium by controlling directly the infinite-dimensional system via finite-dimensional boundary feedback controllers. Our approach is inspired by the fact that dissipative dynamical systems are known to possess finite-dimensional determining parameters (e.g., see \cite{BaVi,FoPr,Hale,Rob,Temam}). We combine this theory with the backstepping method by choosing the so-called \emph{target} model in a novel way.  However, the advantage of our strategy is that, it allows us to gain exponential stabilization of zero equilibrium with a desired decay rate, while the instability level of the model is allowed to be arbitrarily large. This in turn amounts to say that the decay rate can be chosen arbitrarily large, also known as rapid stabilization. We would also like to note that the method in our present study is applicable to various kind second-order dissipative nonlinear PDEs.

It is remarkable that recently a nice strategy, called late-lumping, that combines the method of backstepping with a finite-dimensional controller was introduced, see for instance \cite{Aur19} and \cite{Woit17}  for its application to the heat equation.   There are certain important differences when the methodology of these papers are compared with ours although the essence of the method in \cite{Woit17} is in the same spirit with ours. The method of \cite{Woit17} uses projection of the state onto general finite-dimensional subspaces, which are not necessarily spans of associated eigenfunctions as in this paper. The closed-loop system was not explicitly stated and no stability analysis was given in \cite{Woit17}, but still certain stability features of controlled solutions, at least for the case of a modal basis, can be deduced from their work. One of our main motivations in this paper is to transform the closed-loop system into a target model with homogeneous boundary conditions amenable to a Lyapunov analysis that yields a precise information on the sufficient number of Fourier modes for the desired purpose (stabilization with either a prescribed or an unprescribed rate). On the other hand, in aforementioned works, authors use the standard backstepping transformation as in \cite{krsticbook}.  Therefore, it does not involve the projection operator inside the integral term and this contrasts with our transformation (see equation \eqref{bt} below), which is a priori given with the projection $P_N$ in its definition.  A precise or optimal estimate for the sufficient number of modes is not provided in the late-lumping approach.

In addition, the late-lumping approach is based on a crucial assumption that the finite-dimensional controller uniformly converges to the standard controller that uses the full state of the solution.  Such assumption was verified in \cite[Lemma 5]{Aur19}, however its proof relies on the embedding $H^1(0,L)\hookrightarrow L^\infty(0,L)$. This forces solutions and in turn the initial state to be taken from $H^1(0,L)$, too, for proving decay estimates at $L^2$-level. On the other hand, our method does not rely on the verification of such property and we establish stabilization at $L^2$-level with initial datum taken from $L^2(0,L)$.  Furthermore, we also provide decay rate estimates at $H^1-$level with data from $H^1(0,L)$ - the space used for $L^2$-level decay with late-late lumping approach.

\subsection{Methodology}
Let us now describe our strategy. We first consider the following linearized model associated with \eqref{pde_nl}:
\begin{equation} \label{pde_l}
	\begin{cases}
		u_t - \nu u_{xx} - \alpha u = 0, \quad x\in(0,L), t  > 0, \\
		u(0,t) = 0, \quad u(L,t) = g(t), \\
		u(x,0) = u_0(x).
	\end{cases}
\end{equation}
The steps of our strategy are given below:\\
\textbf{\textit{Step i.}} We first design the following target model:
\begin{equation} \label{pde_tl}
	\begin{cases}
		w_t - \nu w_{xx} - \alpha w + \mu P_Nw= 0, \quad  x\in(0,L), t > 0, \\
		w(0,t) = w(L,t) = 0,\\
		w(x,0)=w_0(x),
	\end{cases}
\end{equation}
where $w_0$ to be computed in step (iii) below. This target model is chosen to convert the feedback action on the boundary to an interior stabilizing effect (e.g., see \cite{Liu03}). In our case, the feedback involves a finite number of Fourier modes, therefore, we consider a target model with a stabilizing term given by $\mu P_N w$. Indeed, dissipative dynamical systems possess finite-dimensional asymptotic dynamics (e.g., see \cite{Hale,Rob,Temam}) since such systems possess a finite number of determining modes \cite{FoPr}. Therefore, the term $\mu P_Nw$ is capable of preventing large fluctuations or uncontrolled growth of the solution for large $N$ due to the term  $-\alpha w$. We will show that solutions of \eqref{pde_tl} can be steered to zero with any prescribed exponential decay rate provided $\mu > \nu\lambda_1 - \alpha$ and $N$ fulfills a certain criterion. \\
\textbf{\textit{Step ii.}} Next, we introduce our backstepping transformation:
\begin{equation} \label{bt}
	\begin{split}
	u(x,t) &= w(x,t) + \int_0^x k(x,y) P_N w(y,t) dy \\
	&\doteq [(I + \Upsilon_k P_N)w](x,t)\doteq[T_Nw](x,t).
	\end{split}
\end{equation}
In Section \ref{sec_bstrans}, we show that if $k$ is a sufficiently smooth solution of a certain boundary value problem (see \eqref{ker_pde}) posed on the triangular region $\Delta_{x,y} \doteq \left\{(x,y) \in \mathbb{R}^2 : x\in(0,L), y \in (0,x)\right\}$, then \eqref{bt} successfully maps the target model \eqref{pde_tl} to  the linearized model \eqref{pde_l}. Finding the kernel $k$ allows us to write the control input as
\begin{equation} \label{ctrl_input}
	g(t) = \int_0^L k(L,y) P_N w(y,t) dy.
\end{equation}
\eqref{ctrl_input} will be rewritten in the form of a feedback using the inverse operator discussed in the next step.\\
\textbf{\textit{Step iii.}} We prove that the transformation $I + \Upsilon_kP_N$ is bounded invertible if $(\mu,N)$ is an admissible decay rate-mode pair (see Definition \ref{ratemodepair} in Section \ref{sec_bstrans}). Bounded invertibility and step (ii) imply that the decay properties of the target model \eqref{pde_tl} are inherited by the linear plant \eqref{pde_l}. Moreover, the inverse  operator can be expressed as $I - \Phi_N$, where $\Phi_N : L^2(0,L) \to H^\ell(0,L)$ is a bounded linear operator. Thus, the boundary controller \eqref{ctrl_input} can be expressed as
\begin{equation}\label{controller}
	g(t) = \int_0^L k(L,y) \Gamma[P_N u](y,t)dy,
\end{equation}
where $\Gamma \doteq (I - P_N \Phi_N): H^\ell(0,L) \to H^\ell(0,L)$. The initial state of the target model takes the form $w_0 \doteq (I - \Phi_N) u_0$.

We use the same backstepping transformation to transform  the nonlinear model \eqref{pde_nl} into the corresponding nonlinear target model below.
\begin{equation*}
	\begin{cases}
		w_t - \nu w_{xx} - \alpha w + \mu P_Nw + \kappa Fw = 0, \quad x\in(0,L), t > 0, \\
		w(0,t) = w(L,t) = 0, \\
		w(x,0) = w_0(x),
	\end{cases}
\end{equation*}
where $Fw = (I - \Phi_N) \left[((I + \Upsilon_kP_N)w)^3\right].$ We will see in Section \ref{sec_wpstab_nonlin} that choosing $\mu > \nu \lambda_1 - \alpha$ and $N$ sufficiently large, solutions of the above nonlinear target model tend to zero at a prescribed exponential decay rate, provided that $\|w_0\|_{H^\ell(0,L)}$ is sufficiently small. Now, thanks to steps (ii)-(iii) above, under a smallness assumption on $u_0$, the same decay result will also hold for the nonlinear plant \eqref{pde_nl}.

\subsection{Main results} We first solve the rapid stabilization problem for the linearized model and prove the theorem below. In our main results, $(\mu,N)$ is always assumed to be an admissible decay rate-mode pair (see Definition \ref{ratemodepair} in Section \ref{sec_bstrans} for details).
\begin{thm}\label{wpstab_lin}
	Let $N\ge 1$, $\nu, \alpha > 0$, $\mu>\alpha - \nu \lambda_1 \geq 0$, $u_0 \in H^\ell(0,L)$, where $\ell = 0$ or $\ell=1$ and $(\mu,N)$ be an admissible decay rate-mode pair. Then, the weak solution $u$ of the linear problem \eqref{pde_l}, which satisfies the feedback control law $u(L,t) = g(P_N u(t))$, belongs to $X_T^\ell$. Moreover, if
	\begin{equation*}
		N > \max \left\{ \frac{\mu}{2\nu\lambda_1} - 1 , \frac{\mu}{\mu + \nu\lambda_1 - \alpha} - 1\right\},
	\end{equation*}
	then $u$ satisfies the decay estimate
	\begin{equation*}
		\|u(t)\|_{H^\ell(0,L)} \leq c_k e^{-\gamma t} \|u_0\|_{H^\ell(0,L)} \quad \text{for }t \geq 0 \text{ a.e.}, 
	\end{equation*}
	where $\gamma = \nu\lambda_1 - \alpha + \mu \left(1 - \frac{1}{N+1}\right)$ and $c_k$ is a nonnegative constant dependent on the kernel $k$ and independent of $u_0$.
\end{thm}

Secondly, we treat the corresponding nonlinear model.
\begin{thm}\label{wpstab_nonlin}
	Let $\nu, \alpha > 0$, $\kappa\in\{-1,1\}$, $\mu>\alpha - \nu \lambda_1 \geq 0$, $d\in(0,1)$, $u_0 \in H^\ell(0,L)$, where $\ell = 0$ or $\ell=1$ and $(\mu,N)$ be an admissible decay rate-mode pair. Let $\epsilon>0$ be any number satisfying $\epsilon\lambda_1\le \max\{\nu-\frac{\mu}{N+1}, 2\gamma\}$ and assume that $\|u_0\|_{H^\ell}$ is sufficiently small in the sense that $\|u_0^{(\ell)}\|_2^2\le d\frac{\sqrt{\epsilon(2\gamma-\epsilon\lambda_1)}\lambda_1^{\ell}}{c_0c_{1}}{\|T_N^{-1}\|_{H^\ell(0,L)\rightarrow H^\ell(0,L)}^{-2}}$, where $\gamma$ is given by \eqref{gamma}, $c_0, c_1$ are defined in \eqref{c0def} and \eqref{c1def}, respectively. If
	\begin{equation*}
		N > \max \left\{\frac{\mu}{2\nu\lambda_1} - 1, \frac{\mu}{\mu + \nu\lambda_1 - \alpha} - 1\right\},
	\end{equation*}
	then the nonlinear problem \eqref{pde_nl} has a global solution $u\in X_T^\ell$, which satisfies the feedback control law $u(L,t) = g(P_N u(t))$, and has the decay estimate
	\begin{equation*}
		\|u(t)\|_{H^\ell(0,L)} \leq c_k e^{-\gamma't} \|u_0\|_{H^\ell(0,L)} \quad \text{for } t \geq 0 \text{ a.e.},
	\end{equation*}
	for any $\gamma'$ with $0<\gamma'<\gamma= \nu\lambda_1 - \alpha + \mu \left(1 - \frac{1}{N+1}\right)$, where $c_k$ is a non-negative constant dependent on the kernel $k$, $d$ and independent of $u_0$.
\end{thm}

\begin{rem}
	The smallness assumption on initial datum is used only for stabilization and it is not a necessary assumption for local existence of solutions.  Smallness condition for stabilization of nonlinear PDEs via backstepping controllers is common, see for instance \cite{CeCo}. This is because the backstepping transformation turns the original nonlinear plant into another plant (target model) in which the monotone structure of the nonlinear term is disrupted.
\end{rem}

Next, we consider the stabilization problem with a certain unprescribed not necessarily large decay rate for the linearized model. We show that if $N$ is the number of unstable modes, it suffices to employ a controller that involves only the first $N$ modes to achieve exponential stabilization.

\begin{thm}\label{wpstab_lin2}
	Let $\nu, \alpha > 0$,  $u_0 \in H^\ell(0,L)$, where $\ell = 0$ or $\ell=1$. Let $(\mu,N)$ be an admissible decay rate-mode pair such that $N$ satisfy $\lambda_N < \frac{\alpha}{\nu} < \lambda_{N+1}$ and $\mu$ satisfy
	\begin{eqnarray}\label{mucond}
		&2(\alpha-\nu\lambda_1) \left(1 - \frac{1}{(N+1)^2}\right)^{-1} < \mu < 2\nu \lambda_{N+1}.
	\end{eqnarray}
	Then, the solution $u$ of linear problem \eqref{pde_l}, which satisfies the feedback control law $u(L,t) = g(P_N u(t))$, belongs to $X_T^\ell$ and has the decay estimate
	\begin{equation*}
		\|u(t)\|_{H^\ell(0,L)} \leq c_k e^{-\rho t} \|u_0\|_{H^\ell(0,L)} \quad \text{for } t \geq 0 \text{ a.e.},
	\end{equation*}
	where $\rho = \nu\lambda_1 - \alpha + \frac{\mu}{2}(1 - \frac{1}{(N+1)^2})$ and, $c_k$ is a nonnegative constant depending on $k$ and independent of $u_0$.
\end{thm}

\begin{rem} The set of $\mu$ that satisfies \eqref{mucond} is nonempty as shown in Section \ref{sec_wpstab2}. Moreover, $\rho$ is strictly positive by construction. Similar remarks also apply to Theorem \ref{wpstab_nonlin2} below, where we extend the above result to the nonlinear model.
\end{rem}

\begin{thm}\label{wpstab_nonlin2}
	Let  $\nu, \alpha > 0$, $\kappa\in\{-1,1\}$, $d\in (0,1)$ and $u_0 \in H^\ell(0,L)$, $\ell = 0$ or $\ell=1$, where $\|u_0\|_{H^\ell}$ is sufficiently small in the sense of Theorem \ref{wpstab_nonlin}. Let $(\mu,N)$ be an admissible decay rate-mode pair such that $N$ satisfy $\lambda_N < \frac{\alpha}{\nu} < \lambda_{N+1}$ and $\mu$ satisfies
	\begin{equation*}
		2(\alpha-\nu\lambda_1) (1 - \frac{1}{(N+1)^2})^{-1} < \mu < 2\nu \lambda_{N+1}.
	\end{equation*}
	Then, for any $\rho'$ with $0<\rho^\prime<\rho= \nu\lambda_1 - \alpha + \frac{\mu}{2}(1 - \frac{1}{(N+1)^2})$, the nonlinear problem \eqref{pde_nl} has a solution $u\in X_T^\ell$, which satisfies the feedback control law $u(L,t) = g(P_N u(t))$, and has the decay estimate
	\begin{equation*}
		\|u(t)\|_{H^\ell(0,L)} \leq c_k e^{-\rho' t} \|u_0\|_{H^\ell(0,L)} \quad \text{for } t \geq 0 \text{ a.e.}, 
	\end{equation*}
	where $c_k$ is a nonnegative constant depending on $k$, $d$, and independent of $u_0$.
\end{thm}

\subsection{Orientation}
This paper consists of six sections. In Section \ref{sec_bstrans}, we find sufficient conditions for the backstepping kernel so that the corresponding backstepping transformation maps the linearized plant to a desired target model. Using these conditions, we obtain an explicit representation of the kernel. Then, we prove the invertibility of the backstepping transformation with a bounded inverse. In Sections \ref{sec_wpstab_lin} and \ref{sec_wpstab_nonlin}, we answer the question stated in Problem \ref{prob}, and in Section \ref{sec_wpstab2} we answer the question stated in Problem \ref{prob2}. Finally, in Section \ref{numerics}, we present some numerical simulations verifying our theoretical results.

\section{Backstepping transformation: existence and bounded invertibility} \label{sec_bstrans}
\subsection{Kernel}
Here, we state the sufficient conditions for the kernel so that the  backstepping transformation maps the target model to the linearized plant. Adapting arguments given in \cite[Section 4]{krsticbook} to the operator in \eqref{bt}, it follows that $k$ is a suitable kernel if it solves the boundary value problem
\begin{equation} \label{ker_pde}
	\begin{cases}
		\nu (k_{xx} - k_{yy}) + \mu k = 0, \quad  (x,y) \in \Delta_{x,y}, \\
		k(x,0) = 0, \quad k(x,x) = -\frac{\mu x}{2\nu}, \quad x \in (0,L),
	\end{cases}
\end{equation}
on $\Delta_{x,y}$. Moreover, the solution can be expressed explicitly as
\begin{equation*}
	\label{ksolrep}k(x,y) = -\frac{\mu y}{2 \nu} \sum\limits_{m=0}^\infty \left(-\frac{\mu}{4 \nu}\right)^m \frac{(x^2 - y^2)^m}{m! (m+1)!},
\end{equation*}
which is a uniformly and absolutely convergent series on $\overline{\Delta_{x,y}}$, closure of $\Delta_{x,y}$. It is important that $P_N$ and $\frac{d^2}{dy^2}$ commutes in derivation of \eqref{ker_pde}.

\subsection{Invertibility}
In this section, we discuss the bounded invertibility of the backstepping transformation on $H^\ell(0,L)$, $\ell = 0,1,2$. 

At first, we argue that the invertibility does not necessarily hold for every choice of $\mu>0$ in \eqref{ker_pde}. To infer this, let $\Upsilon_k : H^\ell (0,L) \to H^\ell (0,L)$ be the integral operator defined by $(\Upsilon_k \psi)(x) \doteq \int_0^x k(x,y) \psi(y) dy$. Let $N=1$ as an example and consider the backstepping transformation $I + \Upsilon_k P_N=I + \Upsilon_k P_1$, which uses a single Fourier mode.  Let $\mu>0$ be such that $(e_1,\Upsilon_ke_1)_2=-1$ (such $\mu$ exists by intermediate value theorem). Let $\varphi$ be a function in $H^\ell(0,L)$ with the property $\hat{\varphi}_1=(\varphi,e_1)_2\neq 0$. We claim that there is no $\psi$ with $\varphi = (I + \Upsilon_k P_1)\psi$ because if it is the case then setting  $v = \Upsilon_k P_1\psi$, we see that $v$ must satisfy $v = \Upsilon_k P_1 (\varphi - v)=(\hat \varphi_1 - \hat{v}_1) \Upsilon_ke_1$. Taking $L^2-$inner product of both sides with $e_1$ and using $(e_1,\Upsilon_ke_1)_2=-1$, it follows that $0=-\hat \varphi_1\neq 0$, which is a contradiction. Hence, $(I + \Upsilon_k P_1)$ cannot be invertible if $\mu$ (equivalently the associated kernel $k=k(x,y;\mu)$) is such that  $(e_1,\Upsilon_ke_1)_2=-1$.

Note however, if $\mu>0$ is such that $(e_1,\Upsilon_ke_1)_2\neq -1$, then given $\varphi$, the function $\psi=\varphi-v$, where $v=\hat{v}_1\Upsilon_ke_1$ with $\displaystyle \hat{v}_1=\frac{\hat{\varphi}_1}{1+(e_1,\Upsilon_ke_1)_2}$ satisfies $(I + \Upsilon_k P_1)\psi=\varphi$.  The preimage
\begin{equation*}
	\varphi-v=\varphi-\frac{\hat{\varphi}_1}{1+(e_1,\Upsilon_ke_1)_2}\Upsilon_ke_1
\end{equation*}
leads to a well-defined inverse operator in the form $I-\Phi_1$, where $$\displaystyle\Phi_1\varphi=\frac{\hat{\varphi}_1}{1+(e_1,\Upsilon_ke_1)_2}\Upsilon_ke_1.$$

Now, let us consider the case of two modes ($N=2$) and the transformation $I+\Upsilon_kP_2$. Suppose $\mu>0$ is such that $(e_1,\Upsilon_ke_1)_2\neq -1$ so that the operator $\Phi_1$ given in above paragraph is well-defined. Let us write $\varphi = (I + \Upsilon_k P_2)\psi$ and set $v = \Upsilon_k P_2\psi$. Then, we have
\begin{equation*}
	\begin{split}
		v &= \Upsilon_k P_{2}(\varphi - v)=\Upsilon_k P_{2}\varphi-\Upsilon_k P_{1}v-\hat{v}_2\Upsilon_k e_2,
	\end{split}
\end{equation*}
from which it follows that $(I+\Upsilon_k P_{1})v=\Upsilon_k P_{2}\varphi - \hat{v}_{2}\Upsilon_k e_{2}.$ Now, applying $I-\Phi_1$ to both sides, we get 
\begin{equation*}
	v=(I - \Phi_1)[\Upsilon_k P_{2}\varphi] - \hat{v}_{2}(I - \Phi_1)[\Upsilon_k e_{2}].
\end{equation*}
Taking inner product of above expression with $e_2$ and rearranging the terms, we obtain
\begin{equation*}
	\hat v_{2}=\frac{\left((I - \Phi_1)[\Upsilon_k P_{2}\varphi],e_{2}\right)_2}{1+\left((I-\Phi_1)[\Upsilon_k e_{2}],e_{2}\right)_2},
\end{equation*} provided $\mu>0$ is also such that $\left((I-\Phi_1)[\Upsilon_k e_{2}],e_{2}\right)_2\neq -1.$ Under the aforementioned two conditions on $\mu$, the inverse takes the form $I-\Phi_2$, where 
\begin{equation*}
	\Phi_{2}\varphi=(I - \Phi_1)[\Upsilon_k P_{2}\varphi]
	-\frac{\left((I - \Phi_1)[\Upsilon_k P_{2}\varphi],e_{2}\right)_2}{1+\left((I-\Phi_1)[\Upsilon_k e_{2}],e_{2}\right)_2} 
	\times(I - \Phi_1)[\Upsilon_k e_{2}].
\end{equation*}
The above argument can be extended to any $N>2$ via recursion (whose details are given in Step 3 in the proof of Lemma \ref{invlem} below). The following definition plays a crucial role in order for such recursive argument to work:
\begin{defn}[Admissible decay-mode pair]\label{ratemodepair}
	Let $\mu>0$ and $N\in\mathbb{Z}_+$. Then, $(\mu,N)$ is said to be an \textit{admissible finite-dimensional backstepping decay rate-mode pair} if
	\begin{equation*}
		\left((I-\Phi_{j-1})[\Upsilon_k e_{j}],e_{j}\right)_2\neq -1,
	\end{equation*}
	for $1\le j\le N$, where $\Phi_0 \doteq 0$ and
	\begin{equation*}
		\Phi_{j}\varphi=(I - \Phi_{j-1})[\Upsilon_k P_{j}\varphi]
		-\frac{\left((I - \Phi_{j-1})[\Upsilon_k P_{j}\varphi],e_{j}\right)_2}{1+\left((I-\Phi_{j-1})[\Upsilon_k e_{j}],e_{j}\right)_2} \times(I - \Phi_{j-1})[\Upsilon_k e_{j}],
	\end{equation*}
	for $1\le j\le N$.
\end{defn}

\begin{lem} \label{invlem}Let $\ell\in\{ 0,1 ,2\}$ and $(\mu,N)$ be an admissible decay-mode pair. Then,
	$T_N = I + \Upsilon_k P_N : H^\ell(0,L) \to H^\ell(0,L)$, is bounded invertible. Moreover $T_N^{-1}$ can be written as $I - \Phi_N$, where $\Phi_N : L^2(0,L) \to H^\ell(0,L)$ is linear bounded.
\end{lem}

\begin{proof}
	We write the backstepping transformation in operator form as $\varphi = (I + \Upsilon_k P_N)\psi.$ Set $v = \Upsilon_k P_N\psi$. Then, we have $\psi =  \varphi - v$ and we get
	\begin{equation}\label{vequation}
		v = \Upsilon_k P_N(\varphi - v).
	\end{equation}
	Given $\psi\in H^\ell(0,L)$, we have $(I + \Upsilon_k P_N)\psi\in H^\ell(0,L).$ Therefore, we have the inclusion $R(I + \Upsilon_k P_N)\subset H^\ell(0,L).$
	We will prove the invertibility with induction on $N$. \\
	
	\emph{Step 1:} We first consider the case $N=1$. \\
	Using \eqref{vequation} with $N=1$, we write $v = \Upsilon_k P_1 (\varphi - v)$. Note that $P_1\varphi=\hat \varphi_1 e_1$ and $P_1v=\hat v_1 e_1$, where $\hat \varphi_1=(\varphi,e_1)_2$ and $\hat{v}_1=(v,e_1)_2$. Using linearity of $\Upsilon_k$, we get \begin{equation}
		\label{veqnew}v=\Upsilon_k(\hat \varphi_1-\hat v_1)e_1= (\hat \varphi_1 - \hat{v}_1) \Upsilon_ke_1.
	\end{equation}
	Now, we look for a solution in the form $\label{vform} v=\alpha_1 \Upsilon_k e_1.$ Note that then $\hat v_1=(e_1,\alpha_1 \Upsilon_k e_1)_2=\alpha_1\beta_1$, where $\beta_1 = \int_{0}^{L}e_1(s) [\Upsilon_k e_1](s) ds.$
	Using this equality in \eqref{veqnew}, we find that $\alpha_1$ should satisfy $\label{alphajsolve}\alpha_1 \Upsilon_ke_1 = (\hat \varphi_1 - \alpha_1 \beta_1)\Upsilon_k e_1.$ This equation will hold with the choice $\alpha_1(1+\beta_1)=\hat\varphi_1$. Since $(\mu,N)$ is assumed to be an admissible decay rate-mode pair, we have $\beta_1\neq -1$, so that we can take $\alpha_1=\frac{\hat\varphi_1}{1+\beta_1}.$ This shows surjectivity of $I + \Upsilon_k P_1$. Note that we in particular have $v\in H^\ell(0,L)$ (indeed better than this) so that $\psi=\varphi-v\in H^\ell(0,L).$
	
	Note we have $$\psi=\varphi-v=\varphi-\frac{\hat\varphi_1}{1+\beta_1}\Upsilon_k e_1 = \varphi - \frac{1}{1 + \beta_1} \Upsilon_k P_1 \varphi.$$
	Now, we set $\label{phiN} \Phi_1 \varphi\equiv \frac{1}{1 + \beta_1} \Upsilon_k P_1 \varphi.$ It is easy to verify that $I-\Phi_1$ is both a right inverse and a left inverse for $I + \Upsilon_k P_1$ and moreover $\Phi_1: L^2(0,L) \to H^\ell(0,L)$ is a linear bounded operator. Therefore, $(I + \Upsilon_k P_1)^{-1}$ exists and is given by $(I + \Upsilon_k P_1)^{-1}{\varphi}=(I-\Phi_1)\varphi$. \\
	
	\emph{Step 2:} We assume that there exists some $K\ge 1$ such that the statement of the lemma holds true for $N=K$.\\
	
	\emph{Step 3:} Now, we claim that the statement of the lemma must also be true for $N=K+1$. Replacing $N$ by $K+1$ we have
	\begin{equation} \label{invlem_vKplus1}
		\begin{split}
			v = \Upsilon_k P_{K+1}(\varphi - v) =\Upsilon_k P_{K+1}\varphi-\Upsilon_k P_{K}v-\Upsilon_k E_{K+1}v,
		\end{split}
	\end{equation} where we used $P_{K+1}=P_K+E_{K+1}$ with $E_{K+1}$ being the projection onto the $(K+1)-$th Fourier sine mode. Rearranging the terms, we get $(I+\Upsilon_k P_{K})v=\Upsilon_k P_{K+1}\varphi - \hat{v}_{K+1}\Upsilon_k e_{K+1},$
	where $\hat{v}_{K+1}$ is the $(K+1)-$th Fourier sine mode of $v$. By using the induction assumption in Step 2, we obtain
	\begin{equation*}
		v=(I - \Phi_K)[\Upsilon_k P_{K+1}\varphi] - \hat{v}_{K+1}(I - \Phi_K)[\Upsilon_k e_{K+1}].
	\end{equation*}
	Taking inner product in both sides of the above equation with $e_{K+1}$, we get
	\begin{equation*}
		\hat v_{K+1}=\frac{\left((I - \Phi_K)[\Upsilon_k P_{K+1}\varphi],e_{K+1}\right)_2}{1+\left((I-\Phi_K)[\Upsilon_k e_{K+1}],e_{K+1}\right)_2}.
	\end{equation*}
	Note that $\hat v_{K+1}$ is well-defined since the denominator does not vanish above due to the assumption that $(\mu,N)$ is an admissible decay rate-mode pair. Therefore, we find
	\begin{equation}\label{PhiKplus1}
		\begin{split}
			v=&(I - \Phi_K)[\Upsilon_k P_{K+1}\varphi] \\
			&-\frac{\left((I - \Phi_K)[\Upsilon_k P_{K+1}\varphi],e_{K+1}\right)_2}{1+\left((I-\Phi_K)[\Upsilon_k e_{K+1}],e_{K+1}\right)_2} (I - \Phi_K)[\Upsilon_k e_{K+1}]\equiv \Phi_{K+1}\varphi.
		\end{split}
	\end{equation}
	This shows that $I + \Upsilon_k P_{K+1}$ is surjective and it has a right inverse given by $I-\Phi_{K+1}$, where $\Phi_{K+1}$ is given in \eqref{PhiKplus1}. Moreover, given $\varphi$, we see from above that $v$ is uniquely determined, and therefore the right inverse is unique. This implies the right inverse is also a left inverse and  $I + \Upsilon_k P_{K+1}$ is invertible and we have
	\begin{equation*}
		\psi=(I + \Upsilon_k P_{K+1})^{-1}{\varphi} =\varphi-v =(I-\Phi_{K+1})\varphi.
	\end{equation*}
	It can be verified that $\Phi_{K+1}: L^2(0,L) \to H^\ell(0,L)$ is a linear bounded operator.  This follows from the definition of $\Phi_{K+1}$ together with the induction assumption in Step 2.
\end{proof}

\begin{rem}
	Replacing $w$ in the control input $g(t) = \int_0^L k(x,L) P_Nw(y,t) dy$ with $w = (I - \Phi_N)u$, and noting $\Phi_Nu=\Phi_NP_Nu$ we see that $g$ can be expressed as $g(t) =\int_0^L k(L,y) \Gamma_N[P_Nu](y,t)dy,$	where $\Gamma_N \doteq I - P_N\Phi_N: H^\ell(0,L) \to H^\ell(0,L)$ is a bounded operator. Consequently, (i) our control input is indeed feedback type and (ii) it is calculated by using only finitely many Fourier sine modes of the state of solution.
\end{rem}

\begin{rem}
	It is important to know that given $\mu>0$, there exists $N\ge 1$ so that $(\mu,N)$ is admissible, and moreover, $N$ satisfies the conditions given in main results.  We claim that there are indeed infinitely many such admissible pairs.  To see this, observe that admissibility of the pair $(\mu,N)$ is by construction equivalent to the invertibility of the operator $T_N=I+\Upsilon_kP_N$ introduced in this section.  It is easy to see that the sequence of operators $T_N$ converges uniformly in operator norm to the operator $T=I+\Upsilon_k$ as $N\rightarrow \infty$, which is the standard backstepping transformation and is well-known to be invertible, see e.g. \cite{krsticbook} for explicit construction of $T^{-1}$, see \cite{Liu03} and \cite{OzBa} for an abstract proof leading to existence of $T^{-1}$ on Sobolev spaces. Now, given $\mu>0$, if there were not infinitely many admissible $(\mu,N)$, there would exist some $N_0$ such that $(\mu,N)$ would not be admissible for $N\ge N_0$. This would imply  for $N\ge N_0$, $0\in \sigma(T_N)$ (spectrum of $T_N$). However, $T_N\rightarrow T$ uniformly, therefore every point in  $\cap_{N\ge N_0}\sigma(T_N)$ is in the spectrum of $T$. Namely, we deduce $0\in \sigma(T)$. But then $T$ cannot be invertible, a contradiction. Hence, there must exist infinitely many admissible decay rate-mode pairs associated with each $\mu>0$.
\end{rem}

\section{Linear theory} \label{sec_wpstab_lin}
In this section, we prove well-posedness of the linearized model \eqref{pde_l} and establish uniform decay rate estimates for its solutions. Our analysis will be first carried out through the target model \eqref{pde_tl} and then, we will use boundedness of the backstepping transformation and its inverse to deduce that the wellposedness and decay results that we obtain for \eqref{pde_tl} are inherited by \eqref{pde_l}.

Existence of a local solution of \eqref{pde_tl} can be obtained rigorously by using the standard Galerkin method. We will only give a proof showing the uniform exponential decay of solutions of \eqref{pde_tl} at topological levels $H^\ell(0,L)$ with $\ell=0$ and $\ell=1$, which will also imply global existence and uniqueness of solutions. This will be done formally by using the multiplier method. Indeed, one first construct approximate regular solutions in the form $w^m(t)=\sum_{k=1}^m d_{m}^{k}(t)e_k\in C([0,T];H_0^\ell(0,L))\cap L^2(0,T;H^{\ell+1}(0,L))$ via Galerkin method, apply the multiplier technique with these solutions and then use a limiting argument to obtain the same type of final estimates for the soughtafter weak solution.  This procedure is rather standard and will be omitted.

Let us first consider the case $u_0\in L^2(0,L)$. This implies $w_0=(I - \Upsilon_kP_N)^{-1}u_0\in L^2(0,L)$ thanks to the bounded invertibility of the backstepping transformation on $L^2(0,L)$ (see Lemma \ref{invlem}). Now, taking the  $L^2-$inner product of the main equation in \eqref{pde_tl} by $2w$, we get
\begin{equation} \label{l2-inner}
	\frac{d}{dt}\|w(t)\|^2 + 2\nu \|w_x(t)\|^2 - 2\alpha \|w(t)\|^2 \\ 
	+ 2\mu \int_0^L w(x,t) P_Nw(x,t) dx = 0.
\end{equation}
Using the Cauchy-Schwarz inequality and Cauchy's inequality with $\epsilon_1 > 0$, we see that the last term at the left hand side of \eqref{l2-inner} is bounded from below as
\begin{equation*}
	\begin{split}
	2\mu \int_0^L w(x,t) P_Nw(x,t) dx =& -2\mu \int_0^L w(x,t) \left(w - P_Nw\right)(x,t) dx + 2\mu\|w(t)\|^2 \\
	\geq& (2\mu - \epsilon_1\mu) \|w(t)\|^2 - \frac{\mu}{\epsilon_1} \|\left(w-P_Nw\right)(t)\|^2.
	\end{split}
\end{equation*}
Using the above bound in \eqref{l2-inner}, we get
\begin{equation*}
	\frac{d}{dt}\|w(t)\|^2 + 2\nu \|w_x(t)\|^2 - 2\alpha \|w(t)\|^2
	+ (2\mu - \mu\epsilon_1) \|w(t)\|^2 - \frac{\mu}{\epsilon_1} \|\left(w-P_Nw\right)(t)\|^2 \leq 0.
\end{equation*}
Employing the Poincaré type inequality, we obtain
\begin{equation}
	\label{l2h1} \frac{d}{dt}\|w(t)\|^2 + 2\left(\nu - \frac{\mu}{2\epsilon_1 \lambda_{1} (N+1)^2}\right) \|w_x(t)\|^2 + (2\mu - \mu\epsilon_1 - 2\alpha)\|w(t)\|^2 \leq 0.
\end{equation}
For $\epsilon_1 > 0$ satisfying
\begin{equation} \label{poincare}
	\nu - \frac{\mu}{2\epsilon_1 \lambda_{1}(N+1)^2} > 0,
\end{equation}
we apply the Poincaré inequality to the second term at the left hand side of \eqref{l2h1}:
\begin{equation} \label{l2-ineq}
	\frac{d}{dt}\|w(t)\|^2 + 2 \biggl[\nu\lambda_1 - \alpha	+ \mu \left(1 - \frac{\epsilon_1}{2} - \frac{1}{2 \epsilon_1 (N+1)^2}\right) \biggr]\|w(t)\|^2 \leq 0.
\end{equation}
To maximize the stabilizing effect one must choose $\epsilon_1 = \frac{1}{N+1}$. Note that \eqref{poincare}  with $\epsilon_1 = \frac{1}{N+1}$ will hold provided
\begin{equation} \label{cond_N1}
	N > \frac{\mu}{2\nu\lambda_1} - 1.
\end{equation}
Therefore, under condition \eqref{cond_N1} on $N$, we get
\begin{equation} \label{l2decay_ineq}
	\frac{d}{dt}\|w(t)\|^2 + 2 \gamma \|w(t)\|^2 \leq 0,
\end{equation}
where
\begin{equation} \label{gamma}
	\gamma = \nu\lambda_1 - \alpha + \mu (1 - {1}/{(N+1)}).
\end{equation}
In order to have decay, we must have $\gamma > 0$ which imposes
\begin{equation} \label{cond_N2}
	N > \frac{\mu}{\mu + \nu\lambda_1 - \alpha} - 1.
\end{equation}
Finally, integrating \eqref{l2decay_ineq}, we deduce 
\begin{equation} \label{l2decayrate}
	\|w(t)\| \leq e^{-\gamma t} \|w_0\|, \quad \text{ for } t \geq 0 \text{ a.e.},
\end{equation}
under the assumption that
\begin{equation} \label{cond_N}
	N > \max \left\{ \frac{\mu}{2\nu\lambda_1} - 1 , \frac{\mu}{\mu + \nu\lambda_1 - \alpha} - 1\right\}.
\end{equation}
Next, we obtain uniform estimates for solutions of \eqref{pde_tl} in the case $u_0\in H^1(0,L)$.  We take $L^2-$inner product of the main equation in \eqref{pde_tl} by $-2w_{xx}$ and deduce
\begin{equation} \label{h1-inner}
	\frac{d}{dt}\|w_x(t)\|^2 + 2\nu \|w_{xx}(t)\|^2 - 2\alpha \|w_x(t)\|^2 - 2\mu \int_0^L w_{xx}(x,t) P_Nw(x,t) dx = 0.
\end{equation}
Applying the Cauchy-Schwarz inequality and then Cauchy's inequality with $\epsilon_2 > 0$, the last term at the left hand side of the above identity can be bounded from below as
\begin{equation} \label{h1est-1}
	\begin{split}
		-2\mu \int_0^L w_{xx}(x,t) P_Nw(x,t) dx
		=& 2\mu \int_0^L w_{xx}(x,t) (w-P_Nw)(x,t) dx \\
		&- 2\mu \int_0^L w_{xx}(x,t) w(x,t) dx \\
		\geq& -\mu\epsilon_2 \|w_{xx}(t)\|^2 - \frac{\mu}{\epsilon_2} \|(w-P_Nw)(t)\|^2 \\
		&+ 2\mu  \|w_x(t)\|^2.
	\end{split}
\end{equation}
Next we use the Poincaré type inequality to get
\begin{equation*}
	-2\mu \int_0^L w_{xx}(x,t) P_Nw(x,t) dx \geq -\mu\epsilon_2 \|w_{xx}(t)\|^2
	+ \left(2\mu - \frac{\mu}{\epsilon_2\lambda_1 (N+1)^2}\right)\|w_x(t)\|^2.
\end{equation*}
Combining the above estimate with \eqref{h1-inner}, we obtain
\begin{equation} \label{h1-inner-2}
	\frac{d}{dt}\|w_x(t)\|^2 + 2\left(\nu - \frac{\mu \epsilon_2}{2}\right) \|w_{xx}(t)\|^2 
	+ 2 \left(\mu - \alpha - \frac{\mu}{2 \epsilon_2\lambda_1 (N+1)^2}\right)\|w_x(t)\|^2 \leq 0.
\end{equation}
Now, for a given $\mu$, we can find $\epsilon_2>0$ such that
\begin{equation} \label{poincareh}
	\nu - \frac{\mu \epsilon_2}{2} > 0.
\end{equation}
For this choice of $\epsilon_2$, we apply the Poincaré inequality to the second term at the left hand side of \eqref{h1-inner-2} and get
\begin{equation} \label{h1-inner-3}
	\frac{d}{dt}\|w_x(t)\|^2 + 2\biggl[\nu \lambda_1 - \alpha
	+ \mu \left(1 - \frac{\epsilon_2 \lambda_1}{2} - \frac{1}{2\epsilon_2\lambda_1(N+1)^2}\right)\biggr] \|w_x(t)\|^2 \leq 0.
\end{equation}
Observe that the inner parenthesis is maximized for the choice $\epsilon_2 = \frac{1}{\lambda_1 (N+1)}$. Using this choice we rewrite \eqref{h1-inner-3} and get
\begin{equation} \label{h1decay_ineq}
	\frac{d}{dt}\|w_x(t)\|^2 + 2\gamma \|w_x(t)\|^2 \leq 0,
\end{equation}
where $\gamma$ is given by \eqref{gamma}. Note that letting $\epsilon_2 = \frac{1}{\lambda_1 (N+1)}$ in \eqref{poincareh} yields the same condition \eqref{cond_N1} on $N$ as we obtained in the $L^2-$decay case. Finally, combining \eqref{h1decay_ineq} with \eqref{l2decay_ineq}, we obtain the following decay result
\begin{equation*}
	\|w(t)\|_{H^1(0,L)} \leq e^{-\gamma t} \|w_0\|_{H^1(0,L)}, \quad \text{ for } t \geq 0 \text{ a.e.},
\end{equation*}
under condition \eqref{cond_N} on $N$. Finally, we deduce $w\in X_T^\ell$ by using the definition of norm of $X_T^\ell$ once we integrate \eqref{l2decay_ineq} and \eqref{h1decay_ineq} over $(0,T)$ for the cases $\ell=0$ and $\ell=1$, respectively.

The following proposition follows.
\begin{prop}\label{l2h1tarlin_decay}
	Let $\nu, \alpha > 0$, $w_0 \in H^\ell(0,L)$, where $\ell = 0$ or $\ell=1$, $\mu>\alpha - \nu \lambda_1 \geq 0$, $N$ satisfy \eqref{cond_N}, and let $(\mu,N)$ be admissible (Definition \ref{ratemodepair}). Then, the solution $w$ of \eqref{pde_tl}, which belongs to $X_T^\ell$, satisfies
	\begin{equation*}
		\|w(t)\|_{H^\ell(0,L)} \leq e^{-\gamma t} \|w_0\|_{H^\ell(0,L)}, \quad \forall t\geq 0,
	\end{equation*}
	where $\gamma$ is given by \eqref{gamma}.
\end{prop}

\begin{rem}
	If one uses all Fourier sine modes ($N=\infty$), then the main equation of the target model takes the form $\label{fullmode}w_t - \nu w_{xx} - \alpha w + \mu w = 0.$
	Using standard multipliers, one can see that $\|w(t)\|_{H^\ell(0,L)} = \mathcal{O}(e^{-(\nu\lambda_1 - \alpha + \mu) t})$, $t \geq 0$. Hence, the condition $\mu > \alpha - \nu\lambda_1$ is necessary for solutions to decay. Therefore, this condition, which also appears in the above proposition is a natural assumption.
\end{rem}

Using boundedness of the backstepping transformation, we deduce that $u$, the solution of \eqref{pde_l}, also belongs to $X_T^\ell$. Moreoever, from Proposition \ref{l2h1tarlin_decay}, we have
\begin{equation} \label{backtoplant1}
	\begin{split}
		\|u(t)\|_{H^\ell(0,L)} &\leq \left(1 + \|k\|_{H^\ell(\Delta_{x,y})}\right) \|w(t)\|_{H^\ell(0,L)}  \\
		&\le \left(1 + \|k\|_{H^\ell(\Delta_{x,y})}\right) e^{-\gamma t} \|w_0\|_{H^\ell(0,L)}.
	\end{split}
\end{equation}
Due to the invertibility of the backstepping transformation with a bounded inverse, we get
\begin{equation} \label{backtoplant2}
	\|w_0\|_{H^\ell(0,L)} \leq \|T_N^{-1}\|_{H^\ell(0,L) \to H^\ell(0,L)} \|u_0\|_{H^\ell(0,L)} .
\end{equation}
Combining \eqref{backtoplant1} and \eqref{backtoplant2}, we conclude that
\begin{equation} \label{backtoplant3}
	\|u(t)\|_{H^\ell(0,L)} \leq c_k e^{-\gamma t} \|u_0\|_{H^\ell(0,L)}, \quad \forall t \geq 0,
\end{equation}
where $c_k = (1 + \|k\|_{H^\ell(\Delta_{x,y})}) \|T_N^{-1}\|_{H^\ell(0,L) \to H^\ell(0,L)}$ is a nonnegative constant independent of the initial datum. \eqref{backtoplant3} also provides continuous dependence of solutions on the data. Hence, complete the proof of Theorem \ref{wpstab_lin}.

\section{Nonlinear theory} \label{sec_wpstab_nonlin}
In this section, we prove Theorem \ref{wpstab_nonlin}. Note that, the same backstepping transformation that was used in Section \ref{sec_bstrans} relates the following target system posed on $(0,L)\times(0,T)$ to the original nonlinear plant
\begin{equation} \label{pde_tnl}
	\begin{cases}
		w_t - \nu w_{xx} - \alpha w + \mu P_Nw = - \kappa Fw, \\
		w(0,t) = w(L,t) = 0, \\
		w(x,0) = w_0(x),
	\end{cases}
\end{equation}
where $\label{nl_term}Fw = T_N^{-1} \left[(T_Nw)^3\right]$.

Local existence of the solution of \eqref{pde_tnl} on an interval $(0,T)$ for sufficiently small $T$ is carried out via the standard fixed point argument (e.g., see \cite{zheng2004nonlinear}). We detail our analysis at Appendix \ref{app-exist}. In the following part, we prove the exponential stabilization of zero equilibrium to the nonlinear target model by proving exponential decay for solutions of \eqref{pde_tnl} in $H^\ell(0,L)$ for $\ell=0$ and $\ell=1$. This amounts to mean that solutions are uniformly bounded for all $t > 0$ which provides global existence in-time.

Now, let us prove the exponential decay for solutions of \eqref{pde_tnl} in $H^\ell(0,L)$ for $\ell=0$ and $\ell=1$.
\subsubsection{$L^2-$decay}
Let $u_0\in L^2(0,L)$ so that $w_0=T_N^{-1}u_0\in L^2(0,L)$. Taking $L^2-$inner product of \eqref{pde_tnl} by $2w$:
\begin{multline} \label{nonlin_mult}
	\frac{d}{dt}\|w(t)\|^2 + 2\nu \|w_x(t)\|^2 - 2\alpha \|w(t)\|^2 + 2\mu \int_0^L w(x,t) P_Nw(x,t) dx \\
	= -2 \kappa \int_0^L  T_N^{-1} \left[(T_Nw)^3\right]w(x,t) dx.
\end{multline}
Applying Cauchy-Schwarz inequality and using boundedness of $T_N^{-1}$ on $L^2(0,L)$, the term at the right hand side of \eqref{nonlin_mult} is dominated by $2c_0\left\| (T_Nw)^3(t) \right\| \|w(t)\|,$ where
\begin{equation}
	\label{c0def}c_0=\|T_N^{-1}\|_{2\rightarrow 2}.
\end{equation}
Applying arguments similar to those in \eqref{nonlin_loc1.5} in Appendix \ref{apriori} and using Cauchy's inequality with $\epsilon$, we get
\begin{equation}
	\begin{split}
	\label{nonlin_est2}
	2c_0\left\|\left(T_Nw\right)^3(t)\right\| \|w(t)\| &\leq  2c_0c_{1}\|w_x(t)\| \|w(t)\|^3 \\
	&\leq \epsilon\|w_x(t)\|^2 + \frac{c_0^2c_{1}^2}{\epsilon}\|w(t)\|^6,
	\end{split}
\end{equation}
where
\begin{equation}
	\label{c1def}c_1=c^*\|T_N\|_{H^{1}(0,L)\rightarrow H^{1}(0,L)}\|T_N\|_{2\rightarrow 2}^2,
\end{equation} 
with $c^*$ being the Gagliardo-Nirenberg constant. Using 
\begin{equation*}
	2\mu \int_0^L w(x,t) P_Nw(x,t) dx \geq (2\mu - \epsilon_1 \mu) \|w(t)\|^2  - \frac{\mu}{\epsilon_1} \|(w - P_Nw)(t)\|^2,
\end{equation*}
with $\epsilon_1 = 1/(N+1)$, we get
\begin{multline*} \label{nonlin_mult3}
	\frac{d}{dt}\|w(t)\|^2 + (2\nu - \epsilon) \|w_x(t)\|^2 + (2\mu  - \frac{\mu}{N+1} - 2\alpha) \|w(t)\|^2 \\
	- \mu (N+1) \|\left(w-P_Nw\right)(t)\|^2\leq \frac{c_0^2c_{1}^2}{\epsilon} \|w(t)\|^6.
\end{multline*}
Thanks to the Poincaré type inequality, we have
\begin{multline}
	\frac{d}{dt}\|w(t)\|^2 + 2\left(\nu - \frac{\epsilon}{2} - \frac{\mu}{2\lambda_1(N+1)}\right) \|w_x(t)\|^2 \\ + \left(2\mu  -  \frac{\mu}{N+1} - 2\alpha\right) \|w(t)\|^2 \leq \frac{c_0^2c_{1}^2}{\epsilon} \|w(t)\|^6.
\end{multline}
One can find $\epsilon$ such that $\nu - \frac{\mu}{2\lambda_1(N+1)} \geq \frac{\epsilon}{2} > 0$ provided the same condition \eqref{cond_N1} on $N$ given in the linear theory holds.  For such $N$, using Poincaré inequality above, we get
\begin{equation} \label{nonlin_mult7}
	\frac{d}{dt}\|w(t)\|^2 + 2\left(\gamma - \frac{\epsilon \lambda_1}{2}\right) \|w(t)\|^2   \leq \frac{c_0^2c_{1}^2}{\epsilon}  \|w(t)\|^6,
\end{equation}
where $\gamma$ is as in \eqref{gamma}. Now, if $\mu > \alpha -\nu \lambda_1 \geq 0$ and $N$ in addition satisfies \eqref{cond_N2}, then \eqref{nonlin_mult7} implies $L^2-$decay provided $\|w_0\|$ is small (see Proposition \ref{l2h1tarnonlin_decay} below). This follows from Lemma \ref{lem_nonlinh1} by taking $y(t)=\|w(t)\|^2$.

\subsubsection{$H^1-$decay}
Here, we consider the case $u_0\in H^1(0,L)$ so that $w_0=T_N^{-1}u_0\in H^1(0,L)$. We take $L^2-$inner product of the main equation in \eqref{pde_tnl} by $-2w_{xx}$, proceed as in \eqref{h1est-1}-\eqref{h1-inner-2} with $\epsilon_2 = \frac{1}{\lambda_1(N+1)}$, and thereby obtain
\begin{multline} \label{nonlin_multh1}
	\frac{d}{dt}\|w_x(t)\|^2 + 2\left(\nu - \frac{\mu}{2\lambda_1(N+1)}\right) \|w_{xx}(t)\|^2  
	+  2 \left(\mu - \alpha - \frac{\mu}{2 (N+1)}\right)\|w_x(t)\|^2 \\ 
	\leq 2 \kappa \int_0^L  T_N^{-1} \left[(T_Nw)^3\right](t) w_{xx}(t) dx.
\end{multline}
Applying Cauchy-Schwarz inequality and using boundedness of $T_N^{-1}$ on $L^2(0,L)$, the right hand side of \eqref{nonlin_multh1} can be estimated as
\begin{equation}
	\label{nonlin_h1est1}
	2 \kappa \int_0^L  T_N^{-1}\left[(T_Nw)^3\right](x,t) w_{xx}(x,t) dx \leq 2c_0 \left\| (T_Nw)^3(t) \right\| \|w_{xx}(t)\|.
\end{equation}
Proceeding as in \eqref{nonlin_loc1.5}, applying Poincaré inequality and Cauchy's inequality with $\epsilon> 0$, we get 
\begin{equation} \label{nonlin_h1est2}
	\begin{split}
		2c_0 \left\| (T_Nw)^3(t) \right\| \|w_{xx}(t)\|
		\leq& \left(2c_{0}c_1 \|w_x(t)\| \|w(t)\|^2\right) \|w_{xx}(t)\| \\
		\leq& 2c_{0}c_1\lambda_1^{-1} \|w_x(t)\|^3 \|w_{xx}(t)\| \\
		\leq& \frac{c_{0}^2c_1^2\lambda_1^{-2}}{\epsilon}\|w_x(t)\|^6 + \epsilon \|w_{xx}(t)\|^2.
	\end{split}
\end{equation}
Using \eqref{nonlin_h1est1}-\eqref{nonlin_h1est2} in \eqref{nonlin_multh1}, we get
\begin{multline*} \label{nonlin_multh2}
	\frac{d}{dt}\|w_x(t)\|^2 + 2\left(\nu - \frac{\epsilon}{2} - \frac{\mu}{2\lambda_1(N+1)}\right) \|w_{xx}(t)\|^2 \\
	+  2 \left(\mu - \alpha - \frac{\mu}{2 (N+1)}\right)\|w_x(t)\|^2  \leq \frac{c_{0}^2c_1^2\lambda_1^{-2}}{\epsilon}\|w_x(t)\|^6.
\end{multline*}
Again, one can find $\epsilon> 0$ such that $\nu - \frac{\mu}{2\lambda_1(N+1)} \geq \frac{\epsilon}{2} > 0$ provided the same condition \eqref{cond_N1} on $N$ given in the linear theory holds. Thus, assuming \eqref{cond_N1}, we can apply the Poincaré inequality and get 
\begin{equation} \label{nonlin_multh4}
	\frac{d}{dt}\|w_x(t)\|^2 + 2 \left(\gamma - \frac{\epsilon \lambda_1}{2}\right) \|w_x(t)\|^2 \leq \frac{c_{0}^2c_1^2\lambda_1^{-2}}{\epsilon}\|w_x(t)\|^6,
\end{equation}
where $\gamma > 0$ is given by \eqref{gamma}. If $\mu > \nu \lambda_1 - \alpha$, then \eqref{nonlin_multh4} yields $H^1-$decay of solutions. This can be obtained from the following lemma by setting $y(t)=\|w_x(t)\|^2$.
\begin{lem} \label{lem_nonlinh1}
	Let $a,b>0$, $d\in(0,1)$, and $y = y(t)> 0$ satisfy
	\begin{equation} \label{lem_ineqh}
		y^\prime(t) +ay(t) - by^3(t)  \leq 0, \quad t \geq 0.
	\end{equation}
	If $y(0)\le d\sqrt{\frac{a}{b}}$, then $y(t) \le \frac{1}{\sqrt{1-d^2}}y(0)e^{-at}$, $t\ge 0$.
\end{lem}
\begin{proof}
	\eqref{lem_ineqh} is a Bernoulli type differential inequality. Solving this inequality, we obtain
	\begin{eqnarray*}
		&y^2(t) \leq \left(\left(\frac{1}{y^2(0)} - \frac{b}{a}\right) e^{2at} + \frac{b}{a}\right)^{-1}.
	\end{eqnarray*}
	The result follows from the assumption $y(0) \le d\sqrt{\frac{a}{b}}$.
\end{proof}
Applying the above lemma to the $H^\ell-$ estimates, we establish the proposition below for the target system.
\begin{prop}\label{l2h1tarnonlin_decay}
	Let $\nu, \alpha > 0$, $\mu>\alpha - \nu \lambda_1 \geq 0$, $d\in(0,1)$, $w_0 \in H^\ell(0,L)$, $\ell = 0$ or $\ell=1$. Let $\epsilon>0$ be any number satisfying $\epsilon\lambda_1\le \max\{\nu-\frac{\mu}{N+1}, 2\gamma\}$, suppose $\|w_0^{(\ell)}\|_2^2\le d\frac{\sqrt{\epsilon(2\gamma-\epsilon\lambda_1)}\lambda_1^{\ell}}{c_0c_{1}}$ and $N$ satisfies \eqref{cond_N}. Then, the solution of \eqref{pde_tnl}, which belongs to $X_T^\ell$, has the decay estimate
	\begin{equation*} \label{l2tarnonlin_decayy}
		\|w(t)\|_{H^\ell(0,L)} \lesssim e^{-\left(\gamma - \frac{\epsilon \lambda_1}{2}\right) t} \|w_0\|_{H^\ell(0,L)}, \text{ for } t \geq 0 \text{ a.e.},
	\end{equation*}
	where $\gamma$ is given by \eqref{gamma}, $c_0, c_1$ are defined in \eqref{c0def} and \eqref{c1def}, respectively.
\end{prop}

Proceeding as in \eqref{backtoplant1}-\eqref{backtoplant3}, we conclude that well-posedness and exponential decay results for the target model are also true for the original plant. Hence, we have Theorem \ref{wpstab_nonlin}.

\begin{rem}Recall that $w_0=T_N^{-1}u_0=(I-\Phi_N)u_0$, therefore $$\|w_0\|_{H^{\ell}(0,L)}\le \|T_N^{-1}\|_{H^\ell(0,L)\rightarrow H^\ell(0,L)}\|u_0\|_{H^{\ell}(0,L)}.$$  Hence, the smallness condition of Proposition \ref{l2h1tarnonlin_decay} will hold provided $$\|u_0^{(\ell)}\|_2^2\le d\frac{\sqrt{\epsilon(2\gamma-\epsilon\lambda_1)}\lambda_1^{\ell}}{c_0c_{1}}{\|T_N^{-1}\|_{H^\ell(0,L)\rightarrow H^\ell(0,L)}^{-2}},$$ which characterizes the domain of attraction for the original plant. Note that the above estimate is still implicit, however with the help of the definition of $\Phi_N$, one may be able to find an upper bound for the quantities $c_0$ and $\|T_N^{-1}\|_{H^\ell(0,L) \to H^\ell(0,L)}$, which would yield an explicit estimate for the domain of attraction.  
\end{rem}


\section{Minimal number of Fourier modes for decay} \label{sec_wpstab2}
Section \ref{sec_wpstab_lin} and Section \ref{sec_wpstab_nonlin} focused on the rapid stabilization problem. In this section, we are interested in determining the minimal number of Fourier sine modes of $u$ to gain exponential stabilization of zero equilibrium for some unprescribed and not necessarily large decay rate. To this end, we will assume that we want to gain exponential stabilization with $N$ Fourier sine modes, and we will extract the conditions that the problem parameters must satisfy.  Recall from \eqref{poincare} and \eqref{l2-ineq} that the conditions 
\begin{equation*}
	\mu < 2\epsilon_1 \nu \lambda_1 (N+1)^2, \quad \mu > \frac{\alpha - \nu \lambda_1}{1 - \frac{\epsilon_1}{2} - \frac{1}{2\epsilon_1 (N+1)^2}}
\end{equation*}
must be satisfied by $\mu$. To guarantee that such $\mu$ exists, we must have
\begin{equation} \label{mu-l2}
	\frac{\alpha - \nu \lambda_1}{1 - \frac{\epsilon_1}{2} - \frac{1}{2\epsilon_1 (N+1)^2}} <  2\epsilon_1 \nu \lambda_1 (N+1)^2,
\end{equation}
which holds provided
\begin{equation} \label{alpha_nu}
	\alpha/\nu< \lambda_1 (N+1)^2(2\epsilon_1 - \epsilon_1^2) = \lambda_{N+1}(2\epsilon_1 - \epsilon_1^2).
\end{equation}
Note that for $\lambda_{j+1}>\frac{\alpha}{\nu} > \lambda_j$, where $j \in \{ 1, 2, \dotsc\}$, the first $j$ Fourier modes are unstable while the $(j+1)-$th Fourier mode is stable. Thus, let us maximize the right hand side of \eqref{alpha_nu} in order to maximize the range selection for $\frac{\alpha}{\nu}$. This is done with the choice $\epsilon_1 = 1$, which yields
\begin{equation} \label{instalevel}
	\frac{\alpha}{\nu} < \lambda_{N+1}.
\end{equation}
Similarly, from \eqref{poincareh} and \eqref{h1-inner-3}, we need
\begin{equation} \label{mu-h1}
	\mu < \frac{2\nu}{\epsilon_2}, \quad \mu > \frac{\alpha - \nu \lambda_1}{1 - \frac{\epsilon_2 \lambda_1}{2} - \frac{1}{2 \epsilon_2 \lambda_{1}(N+1)^2}}.
\end{equation}
Combining these inequalities, we deduce that we must have $\frac{\alpha}{\nu} < \frac{2}{\epsilon_2} - \frac{1}{\epsilon_2^2 \lambda_1 (N+1)^2}$ in order to guarantee that $\mu$ satisfying \eqref{mu-h1} exists. We again maximize the right hand side with $\epsilon_2 = \frac{1}{\lambda_1 (N+1)^2}$, and therefore get the same inequality given by \eqref{instalevel}. Consequently, as long as the $(N+1)-$th Fourier mode is stable, we can achieve exponential stabilization of zero equilibrium by using a boundary feedback controller involving only the first $N$ Fourier modes. Now, using the above choices of $\epsilon_1$ and $\epsilon_2$ in \eqref{l2-ineq} and \eqref{h1-inner-3}, respectively, one can see that the decay rate becomes
$
\rho = \nu\lambda_1 - \alpha + \frac{\mu}{2}\left(1 - \frac{1}{(N+1)^2}\right) > 0,
$
where $\mu$ obeys
\begin{eqnarray} \label{mu2}
	&2(\alpha-\nu\lambda_1) \left(1 - \frac{1}{(N+1)^2}\right)^{-1} < \mu < 2\nu \lambda_{N+1}.
\end{eqnarray}
Since conditions \eqref{mu-l2} and \eqref{mu-h1} above are required for the stabilization of both linear and nonlinear target models, the conclusion we obtain here is valid for both. The resuts are stated in Theorem \ref{wpstab_lin2} and Theorem \ref{wpstab_nonlin2}.

\section{Numerical algorithm and simulations} \label{numerics}
Notice that our control input involves the integral operator $\Upsilon_k$ and the inverse backstepping operator $I - \Phi_N$. Thus, before deriving a numerical scheme for our continuous models, we need to express discrete counterparts of these operators.

To this end, let $N_x > 1$ be a positive integer, $1 \leq j \leq i \leq N_x$, and $(x_i,y_j)$ be distinct points of the triangular region $\Delta_{x,y}$, where
$
\delta_x = \frac{L}{N_x - 1}, \quad x_i = (i - 1)\delta_x, \quad y_j = (j - 1)\delta_x.
$
We derive backstepping kernel approximately
$$
k^M(x_i,y_j) = \frac{-\mu y_i}{2 \nu} \sum_{m=0}^M \left(-\frac{\mu}{4 \nu}\right)^m \frac{(x_i^2 - y_j^2)^m}{m! (m+1)!},
$$
by setting $M$ so that the error $\underset{1 \leq j \leq i \leq N_x}{\max}|k^{M+1}(x_i,y_j) - k^M(x_i,y_j)|$ is less than the value \emph{eps}, that represents the double-precision floating-point, around $10^{-16}$. Using $k^M$ and applying composite trapezoidal rule for the integration, discrete counterpart, $\mathbf{\Upsilon}_k^h$, of $\Upsilon_k$ can be expressed by an $N_x-$dimensional square matrix with elements
\begin{equation*}
	(\mathbf{\Upsilon}_k^h)_{i,j} =
	\begin{cases}
		0,  &\text{if }j > i, \\
		\frac{\delta_x}{2}k(x_i,x_j), &\text{if } j = 1\text{ or }j = i, \\
		\delta_x k(x_i,x_j), &\text{else}.
	\end{cases}
\end{equation*}
The discrete counterpart of the operator $P_N$, say $\mathbf{P}_N^h$, can also be expressed by an ${N_x}-$dimensional square matrix $\mathbf{P}_N^h = \delta_x \mathbf{W} \mathbf{W}^T$, where $\mathbf{W}$ is an $N_x \times N$ matrix with elements
$
(\mathbf{W})_{i,n} = \sqrt{\frac{2}{L}} \sin \left(\frac{ n\pi x_i}{L}\right), n = 1, \dotsc, N.
$

Next, we obtain an approximation to $\Phi_N u$ using the iteration 
\begin{equation} \label{invlem_vN1}
	\Phi_{N} u \equiv (I - \Phi_{N-1})[\Upsilon_k P_{N} u] -\frac{\left( (I-\Phi_{N-1})[\Upsilon_kP_N u],e_{N}\right)_2}{1+\left((I-\Phi_{N-1})[\Upsilon_k e_N],e_{N}\right)_2}(I-\Phi_{N-1})[\Upsilon_ke_N].
\end{equation}
Iteration \eqref{invlem_vN1} can be viewed as a mapping, say $\Psi$, from $p-$th level to $(p+1)-$th level, $p = 1, 2, \dotsc, N-1$, since derivation of $\Phi_N u$ via \eqref{invlem_vN1} requires information related with $\Phi_{N-1}$. More precisely one explicitly needs to know $\Phi_{N-1}[\Upsilon_k P_N u]$ and $\Phi_{N-1}[\Upsilon_ke_N]$ to derive $\Phi_N u$. Let us illustrate this situation schematically as follows:
\begin{equation} \label{scheme2}
	\left.
	\begin{aligned}
		&\Phi_{N-1} [\Upsilon_kP_N u] \\
		&\Phi_{N-1} [\Upsilon_ke_N]
	\end{aligned}
	\right\}
	\xlongrightarrow{\Psi} \Phi_N u.
\end{equation}
However, $\Phi_{N-1} [\Upsilon_kP_N u]$ and $\Phi_{N-1} [\Upsilon_ke_N]$ are unknowns as well, since they depend on $\Phi_{N-2}$ due to iteration \eqref{invlem_vN1}. In other words, considering \eqref{invlem_vN1} for $N -1$ and replacing $u$ with $\Upsilon_k P_N u$ or $\Upsilon_ke_N$, we see that we need information about $\Phi_{N-2}[(\Upsilon_k P_{N-1})(\Upsilon_k P_N) u]$ and $\Phi_{N-2}[\Upsilon_k e_{N-1}]$ or $\Phi_{N-2}[(\Upsilon_k P_{N-1})(\Upsilon_k e_N)]$ and $\Phi_{N-2}[\Upsilon_k e_{N-1}]$, respectively. Together with the iteration \eqref{scheme2}, let us illustrate this as the following scheme:
\begin{equation*} \label{scheme3}
	\left.
	\begin{aligned}
		\left.
		\begin{aligned}
			&\Phi_{N-2}[(\Upsilon_k P_{N-1})(\Upsilon_k P_N) u] \\
			&\Phi_{N-2}[\Upsilon_k e_{N-1}]
		\end{aligned}
		\right\} \xlongrightarrow{\Psi}
		&\Phi_{N-1} [\Upsilon_kP_N u] \\
		\left.
		\begin{aligned}
			&\Phi_{N-2}[(\Upsilon_k P_{N-1})(\Upsilon e_N)] \\
			&\Phi_{N-2}[\Upsilon_k e_{N-1}]
		\end{aligned}
		\right\} \xlongrightarrow{\Psi}
		&\Phi_{N-1}[\Upsilon_k e_N]
	\end{aligned}
	\right\} \xlongrightarrow{\Psi} \Phi_N u.
\end{equation*}
The functions lying at the most left end of the above scheme are unknown since they depend on $\Phi_{N-3}$. Iteratively, each backward step to determine $\Phi_j$ gives rise to a new unknown including $\Phi_{j-1}$, $j = 3, \dotsc , N$. Therefore, in order to understand what total information is required to derive $\Phi_N u$, we need to go backward step by step until we reach $\Phi_2$, since it depends on $\Phi_1$, which is already known as
$
\Phi_1 \varphi = \frac{1}{1 + \beta_1} \Upsilon_kP_1 \varphi,
$
where $\beta_1 = \int_0^L e_1(s) [\Upsilon_k e_1](s) ds$. For example, if one needs to find $\Phi_3 u$, then one can see by taking $N = 3$ on \eqref{scheme3} that $(\Upsilon_k P_1)[\Upsilon_ke_2]$, $(\Upsilon_k P_1)(\Upsilon_k P_2)[\Upsilon_k e_3]$ and $(\Upsilon_kP_1)(\Upsilon_kP_2)(\Upsilon_kP_3)u$ are required inputs. In general, to derive $\Phi_N u$ explicitly, we need
\begin{equation} \label{inviter_input1}
	(\Upsilon_k P_1)(\Upsilon_k P_2) \dotsm (\Upsilon_k P_{N-1})(\Upsilon_k P_{N})u
\end{equation}
and
\begin{align} 
	&(\Upsilon_k P_1)[\Upsilon_k e_2], \tag{6.3-2} \label{inviter_input2} \\
	&(\Upsilon_k P_1)(\Upsilon_k P_2)[\Upsilon_k e_3], \tag{6.3-3} \label{inviter_input3} \\
	& \quad\vdots \nonumber\\
	&(\Upsilon_k P_1)(\Upsilon_k P_2) \dotsc (\Upsilon_k P_{j-1})[\Upsilon_k e_j], \tag{6.3-j} \label{inviter_inputj} \\
	& \quad\vdots \nonumber\\
	&(\Upsilon_k P_1)(\Upsilon_k P_2) \dotsm (\Upsilon_k P_{N-1})[\Upsilon_k e_N], \tag{6.3-N} \label{inviter_inputN}
\end{align}
as an input. Algorithm \ref{alg:inviter} gives a numerical construction of this.
\begin{algorithm}[H]
	\begin{algorithmic}[1]
		\REQUIRE $\varphi$
		\IF{$N == 1$}
		\STATE $\mathbf{\Phi_N^h}u \gets \frac{1}{1 + \beta_1} \mathbf{\Upsilon_k^h} \mathbf{P_N^h}\varphi;$
		\ELSIF{$N > 1$}
		\STATE $K_\text{old}(1) \gets \text{expression } \eqref{inviter_input1};$
		\FOR{$j=2 \to N$}
		\STATE $K_\text{old}(j) \gets \text{expression } \eqref{inviter_inputj};$
		\ENDFOR
		\STATE $i \gets N;$
		\FOR{$p = 2 \to i$}
		\STATE $K_\text{new} \gets \mathbf{0}_{i-1}$
		\STATE $K_\text{new}(1) \gets \Psi(K_\text{old}(1),K_\text{old}(2));$
		\FOR{$j = 3 \to i$}
		\STATE $K_\text{new}(j-1) \gets \Psi\left(K_\text{old}(2),K_\text{old}(j)\right);$
		\ENDFOR
		\STATE \text{clear} $K_\text{old};$
		\STATE $K_\text{old} \gets K_\text{new};$
		\STATE $i \gets i - 1;$
		\ENDFOR
		\STATE $\mathbf{\Phi_N^h}\varphi \gets K_\text{old}(1);$
		\ENDIF
	\end{algorithmic}
	\caption{Numerical algorithm for derivation of $\mathbf{\Phi_N^h}\varphi$.}
	\label{alg:inviter}
\end{algorithm}

In view of the above discretization procedure, control input is of the form
\begin{equation} \label{num_ctrl}
	g(u(\cdot,t)) = \int_0^L k(L,y) \mathbf{P}_N^h [(\mathbf{I}_{N_x} - \mathbf{\Phi}_N^h) u(y,t)] dy,
\end{equation}
where $\mathbf{\Phi}_N^h$ is derived by using the above iterative scheme involving finite-dimensional operators $\mathbf{\Upsilon}_k^h$, $\mathbf{P}_N^h$ and $\mathbf{I}_{N_x}$ is the $N_x-$dimensional identity matrix. Note that we evaluate the integral \eqref{num_ctrl} by the composite trapezoidal rule.

Next, we derive discrete schemes for our linear and nonlinear continuous models. Define the finite-dimensional space $\mathrm{X}^{h} \doteq \{\mathbf{\varphi} \, | \, \mathbf{\varphi} = [\varphi_1 \, \dotsm \, \varphi_{N_x}]^T \in \mathbb{R}^{N_x}\}$ with the property that $\varphi_1 = 0$ and $\varphi_{N_x} = g(\varphi)$. Define the operator $\mathbf{A}: \mathrm{X}^h \to \mathrm{X}^h$,
$
\mathbf{A} \doteq -\nu  \mathbf{\Delta} - \alpha \mathbf{I}_{N_x} + \mu \mathbf{P}_N^h,
$
where $\mathbf{\Delta}$ is the $N_x-$dimensional tridiagonal matrix obtained by applying central difference approximation to the second-order derivative.

Let $N_t$ denotes the number of time steps, $T_\text{max}$ is the final time and $\delta_t = \frac{T_\text{max}}{N_t - 1}$. Let $\mathbf{u}^n = [u_1^n  \, \dotsm \, u_{N_x}^n]^T \in \mathrm{X}^h$ be an approximation of the solution $u$ at $n-$th time level. Using Crank-Nicolson time stepping, which yields unconditionally numerical stable results in time, we end up with the following fully discrete problem: For each $n = 0, 1, \dotsc, N_t$,
\begin{equation} \label{disc_nonlin}
	\begin{cases}
		\text{given } \mathbf{u}^n \in \mathrm{X}^h, \text{ find } \mathbf{u}^{n+1} \in \mathrm{X}^h \text{ such that} \\
		\left(\mathbf{I}_{N_x} + \frac{\delta_t}{2} \mathbf{A}\right)\mathbf{u}^{n+1} + \frac{\kappa\delta_t}{2} \left(\mathbf{u}^{n+1}\right)^3 = \mathbf{B}_l^n + \mathbf{B}_{nl}^n , \\
		u^{n+1}_{N_x} = g(\mathbf{u}^{n}),
	\end{cases}
\end{equation}
where
\begin{equation} \label{disc_nonlin_rhs}
	\mathbf{B}_l^n = \left(\mathbf{I}_{N_x} - \frac{\delta_t}{2} \mathbf{A}\right) \mathbf{u}^{n}, \quad\mathbf{B}_{nl}^n = -\frac{\kappa \delta_t}{2}\left(\mathbf{u}^{n}\right)^3.
\end{equation}
If $\kappa = 0$, i.e., the linearized version of the nonlinear model, \eqref{disc_nonlin}-\eqref{disc_nonlin_rhs} leads to linear system of equations in the unknown $\mathbf{u}^{n+1}$, hence can be solved. Regarding the case $\kappa = \mp 1$, \eqref{disc_nonlin}-\eqref{disc_nonlin_rhs} is nonlinear in $\mathbf{u}^{n+1}$. To solve \eqref{disc_nonlin}-\eqref{disc_nonlin_rhs} for $\mathbf{u}^{n+1}$, we first apply the following linearization scheme: Let $\mathbf{u}^{n,p}$ be an approximation to the unknown $\mathbf{u}^{n+1}$ and consider performing the iteration
\begin{equation} \label{iter_nonlin}
	\begin{cases}
		\mathbf{u}^{n,p+1} = \mathbf{u}^{n,p} + \mathbf{du}, \quad p = 0, 1, 2 \dotsc, \\
		\mathbf{u}^{n,0} = \mathbf{u}^n,
	\end{cases}
\end{equation}
together with
\begin{equation} \label{iter_nonlin2}
	u^{n,p+1}(N_x) = g(\mathbf{u}^{n,p}),
\end{equation}
to obtain a better approximation, $\mathbf{u}^{n,p+1}$, until $\underset{1 \leq i \leq N_x}{\max} |\mathbf{du}|$ is small enough. To determine $\mathbf{du}$ in each iteration, we replace $\mathbf{u}^{n+1}$ by $\mathbf{u}^{n,p} + \mathbf{du}$ for the linear term and approximate the nonlinear term $\left(\mathbf{u}^{n+1}\right)^3$ linearly in $\mathbf{du}$ by $\left(\mathbf{u}^{n,p}\right)^3 + 3 \mathbf{du} \left(\mathbf{u}^{n,p}\right)^2$. Then, we get
\begin{equation} \label{disc_nonlin2}
	\left(\mathbf{I}_{N_x} + \frac{\delta_t}{2} \mathbf{A}\right) (\mathbf{u}^{n,p} + \mathbf{du}) 
	\frac{\kappa \delta_t}{2} \left(\left(\mathbf{u}^{n,p}\right)^3 + 3 \mathbf{du} \left(\mathbf{u}^{n,p}\right)^2\right)	= \mathbf{B}_l^n + \mathbf{B}_{nl}^n.
\end{equation}
Now, we reexpress \eqref{disc_nonlin2} in terms of $\mathbf{du}$ and write
\begin{equation} \label{dw_lin}
	\left(\mathbf{I}_{N_x} + \frac{\delta_t}{2} \mathbf{A} + \frac{3\kappa\delta_t}{2} \text{diag} \left(\mathbf{u}^{n,p}\right)^2 \right)\mathbf{du} \\
	= \mathbf{B}_l^n + \mathbf{B}_{nl}^n 
	- \left(\mathbf{I}_{N_x} + \frac{\delta_t}{2} \mathbf{A}\right)\mathbf{u}^{n,p} - \frac{\kappa\delta_t}{2} \left(\mathbf{u}^{n,p}\right)^3.
\end{equation}
Here, $\text{diag} \left(\mathbf{u}^{n,p}\right)^2$ represents $N_x-$dimensional diagonal matrix, where the diagonal elements are $\left(u^{n,p}(x_i)\right)^2$, $i = 1, \dotsc, N_x$. Observe that \eqref{dw_lin} is linear in $\mathbf{du}$.  Consequently for given $\mathbf{u}^{n,p}$, \eqref{dw_lin} can be solved directly to obtain $\mathbf{du}$. Once $\underset{1 \leq i \leq N_x}{\max} |\mathbf{du}|$ is small enough, we stop the iteration \eqref{iter_nonlin}-\eqref{iter_nonlin2} and set $\mathbf{u}^{n+1} = \mathbf{u}^{n,p} + \mathbf{du}$. See Algorithm \ref{alg:numnonlin}.
\begin{algorithm}[H]
	\begin{algorithmic}[1]
		\REQUIRE  Initial state $u_0$.
		\STATE $\text{TOL} \gets \text{eps}$;
		\STATE $\mathbf{u}^1 \gets \mathbf{u_0}$;
		\FOR{$n=1 \to N_t - 1$}
		\STATE $u^{n+1}(1) \gets 0$;
		\STATE $p \gets 0$;
		\STATE $\mathbf{u}^{n,p} \gets \mathbf{u}^n$;
		\STATE $\mathbf{du} = \mathbf{1}_{N_x}$;
		\WHILE{$\text{max}|\mathbf{du}| > \text{TOL}$}
		\STATE \text{Solve \eqref{dw_lin} for }$\mathbf{du}$;
		\STATE $\mathbf{u}^{n,p+1} \gets \mathbf{u}^{n,p} + \mathbf{du}$;
		\STATE
		\begin{varwidth}[t]{\linewidth}
			$u^{n,p+1}(N_x) \gets$\par
			\hskip\algorithmicindent $\text{trapz}(y,k(L,y_j)\mathbf{P}_N^h[(\mathbf{I}_{N_x}-\mathbf{\Phi}_N^h) u^{n,p}(y_j)])$;
		\end{varwidth}
		\STATE $p \gets p + 1$;
		\ENDWHILE
		\STATE $\mathbf{u}^{n+1} = \mathbf{u}^{n,p+1}$;
		\ENDFOR
	\end{algorithmic}
	\caption{Numerical algorithm for nonlinear model.}
	\label{alg:numnonlin}
\end{algorithm}

The following numerical experiment is to verify our theoretical stabilization result for the nonlinear model. Results are obtained by taking $N_x = 1000$ spatial nodes and $N_t = 1000$ time steps.

Consider the nonlinear model
\begin{equation} \label{num2}
	\begin{cases}
		u_t - u_{xx} - 15 u + \kappa u^3 = 0, \quad x\in(0,1), t > 0, \\
		u(0,t) = 0, u(1,t) = g(t), \quad t > 0, \\
		u(x,0) = - \frac{1}{2}\sin(2\pi x) + \sin(3\pi x), \quad x \in (0,1).
	\end{cases}
\end{equation}
We take $\kappa=-1$. Therefore, if there is no control input acting on the model, then solution may grow unboundedly (see Fig. \ref{fig:3d_wocont_nonlin}).

\begin{figure}[!h]
	\centerline{\includegraphics[width=0.6\columnwidth]{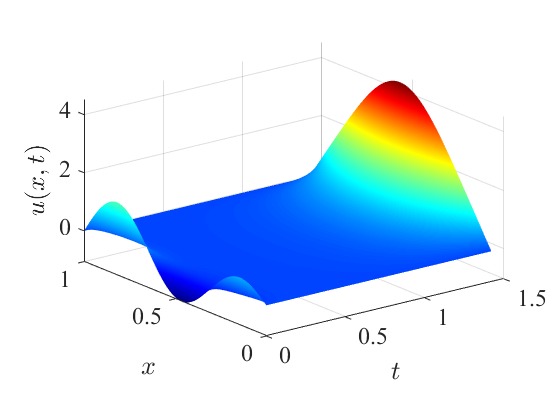}}
	\caption{Time evolution of the solution of uncontrolled nonlinear model \eqref{num2}.}
	\label{fig:3d_wocont_nonlin}
\end{figure}
Let us take $\mu = 15 > \alpha - \nu \lambda_1 = 15 - \pi^2$. Then, choosing $N = 2$ fulfills the condition
$$N = 2> \max \left\{\frac{\mu}{2\nu\lambda_1} - 1, \frac{\mu}{\mu + \nu\lambda_1 - \alpha} - 1\right\} = \max\left\{\frac{15}{2 \pi^2} -1,\frac{15}{\pi^2}- 1 \right\}.$$ With this choice of $\mu$ and $N$, we have $1 + \beta_1 = 1.746 \neq 0$ and $1+\left((I-\Phi_1)[\Upsilon_k e_{2}],e_{2}\right)_2 = 0.845 \neq 0$. Hence $(\mu,N) = (15,2)$ is an admissible decay rate-mode pair and we ensure that Lemma \ref{invlem} holds.

See Fig. \ref{fig:ctrl_nonlin} for the numerical simulation of the stabilized nonlinear model \eqref{num2}.

\begin{figure}[!h]
	\centering
	\begin{subfigure}[b]{0.49\textwidth}
		\centering
		\includegraphics[width=\textwidth]{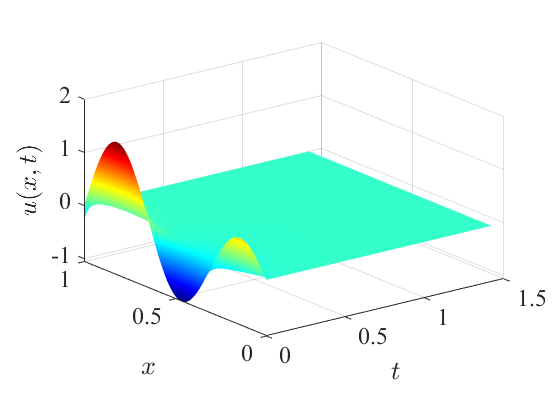}
		\caption{Time evolution of the \\ solution.}
		\label{fig:ctrl_nonlin_3d}
	\end{subfigure}
	\begin{subfigure}[b]{0.49\textwidth}
		\centering
		\includegraphics[width=\textwidth]{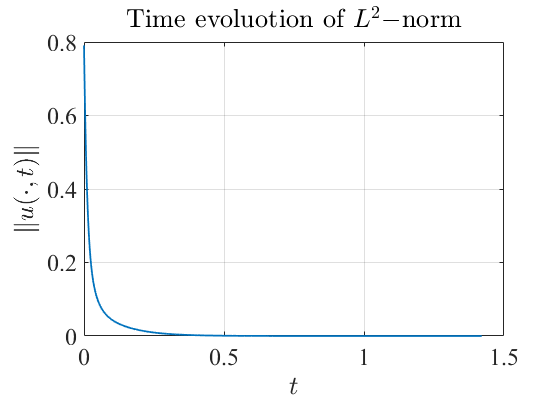}
		\caption{Time evolution of $L^2-$norm of the solution.}
		\label{fig:ctrl_nonlin_l2}
	\end{subfigure}
	\caption{Simulations for the stabilized nonlinear model \eqref{num2}.}
	\label{fig:ctrl_nonlin}
\end{figure}

\subsection*{A few words on potential applications}
Our proposed algorithm suggests an efficient and fast solution for stabilizing heat and more generally a diffusion process through the boundary of a physical medium, due to its rapid stabilization feature that uses only finitely many Fourier modes. The feedback stabilization technique introduced in this work has potential applications across various engineering fields. A classical example is the thermal management in engineering systems.  For example, maintaining stable temperatures in electronic systems, such as chips and processors, and in heat-sensitive sensors used in robotics and automation, is critical to prevent overheating \cite{Moore}. In particular, the nonlinear model \eqref{num2} can be considered as a one-dimensional abstraction (simplification) of thermal changes in a processor with, say constant heat capacity and thermal conductivity, and a nonlinear heat generation (e.g., when $\kappa = -1$) due to power dissipation.  It is observed from the numerical simulation that one can achieve thermal management for such a system by using a boundary feedback stabilizer that uses only finitely many information associated with the state.

We predict our technique can also find use in other fields of engineering. For example, in chemical engineering, controlling reactor stability is essential, e.g., see \cite{Dejan} and the references therein. Often, access to the interior of the medium is limited in chemical reactors. In such cases, e.g., for tubular reactors, an effective control of heat and concentration profiles from the boundary is invaluable.  Similarly, in material production processes like solidification, managing heat at the boundary helps stabilize phase changes \cite{Koga2020}. Another practical application of the proposed method is perhaps the management of heat and ventilation control of road tunnels in civil engineering \cite{Wang2024}.  Stabilization of temperature and air quality in these structures can be achieved by managing heat and pollutants at the boundaries. These diverse examples highlight potential versatility of the suggested algorithm.

\section*{Acknowledgment}
This research was funded by TÜBİTAK 1001 Grant \#122F084. T. Özsarı's research was partially supported by Science Academy's Young Scientist Award (BAGEP).

\appendix

\section{Local well-posedness for the nonlinear target model}
\label{app-exist}
\subsection{Linear non-homogeneous estimates} \label{apriori}
Consider the linear, non-homogeneous problem on $(0,L)\times(0,T)$:
\begin{equation} \label{pde_apriori}
	\begin{cases}
		q_t - \nu q_{xx} - \alpha q + \mu P_Nq = f(x,t), \\
		q(0,t) = q(L,t) = 0,\\
		q(x,0) = q_0(x).
	\end{cases}
\end{equation}
We first prove a linear nonhomogeneous estimate in $X_T^0$ assuming $q_0\in L^2(0,L)$. Taking $L^2-$inner product of the main equation by $2q$ and integrating by parts, we get
\begin{equation} \label{apriori_l2}
	\frac{d}{dt}\|q(t)\|^2 + 2\nu \|q_x(t)\|^2 - 2\alpha \|q(t)\|^2  
	+2\mu \int_0^L q(t) P_Nq(t) dx \leq 2\int_0^L |f(t)||q(t)| dx.
\end{equation}
Integrating above inequality over $[0,t]$ and using Cauchy-Schwarz inequality in $x$, we get 
\begin{equation*}
	\begin{split}
	\min\{1,2\nu\}\cdot\left(\|q(t)\|^2 + \int_0^t\|q_x(s)\|^2ds\right) \le& 2\alpha\int_0^t \|q(s)\|^2ds+2\mu \int_{0}^{t}\|q(s)\|\|P_Nq(s)\|ds\\ 
	&+\|q_0\|^2+2\int_{0}^{t}\|q(s)\|\|f(s)\|ds\\
	\le& \|q_0\|^2+2(\alpha+\mu)\int_0^T \|q(s)\|^2ds\\
	&+2\left(\esssup_{t\in[0,T]}\|q(t)\|\right)\cdot\|f\|_{L^1(0,T;L^2(0,T))}.
	\end{split}
\end{equation*}
It follows that for $t\in [0,T]$ a.e., in view of Cauchy's inequality, for $\epsilon>0$, we get
\begin{multline*}
	\|q(t)\|^2 + \int_0^t\|q_x(s)\|^2ds
	\le C_1\|q_0\|^2+C_2T\esssup_{t\in[0,T]}\|q(t)\|^2
	\\ +{\epsilon}\left(\esssup_{t\in[0,T]}\|q(t)\|^2\right)+C_\epsilon\|f\|_{L^1(0,T;L^2(0,T))}^2,
\end{multline*}
where the constants at the right hand side are given by $C_1=\frac{1}{\min\{1,2\nu\}}$, $C_2=\frac{2(\alpha+\mu)}{\min\{1,2\nu\}}$, and $C_\epsilon=\frac{1}{\epsilon\cdot \min\{1,2\nu\}^2}$.

Now, taking $\esssup$ with respect to $t$ on $[0,T]$, we get
\begin{equation}\label{apest01}
	\|q\|_{X_T^0}\le c_T(\|q_0\|+\|f\|_{L^1(0,T;L^2(0,L))}),
\end{equation} where $c_T=\frac{\max\{C_1,C_\epsilon\}}{1-\epsilon-C_2T}>0$ with $\epsilon$ and $T$ small enough that $\epsilon+C_2T<1$.
Next, we prove a linear non-homogeneous estimate in $X_T^1$ assuming $q_0\in H_0^1(0,L)$. Taking the $L^2-$inner product of the main equation by $-2q_{xx}$,
\begin{multline*} \label{h1aoriori}
	\frac{d}{dt}\|q_x(t)\|^2 + 2\nu \|q_{xx}(t)\|^2 - 2\alpha \|q_x(t)\|^2 \\ 
	- 2\mu \int_0^L q_{xx}(x,t) P_Nq(x,t) dx  \leq 2\int_0^L |f(x,t)||q_{xx}(x,t)| dx.
\end{multline*}
Using arguments similar to those used for \eqref{apest01}, we obtain
\begin{equation}\label{XT1est}
	\|q\|_{X_T^1}\le c_T(\|q_0\|_{H_0^1(0,L)}+\|f\|_{L^1(0,T;L^2(0,L))}).
\end{equation}

\subsection{Contraction} We assume $u_0\in H^{\ell}(0,L)$ where either $\ell=0$ or $\ell=1$. Therefore, by bounded invertibilility of the backstepping transformation, we have  $w_0\in H^{\ell}(0,L)$ for the same value of $\ell$. We set the solution map
\begin{eqnarray} \label{wp_map2}
	&w = \Psi z(t) \doteq S(t)w_0 + \kappa \int_0^t S(t - \tau) Fz(\tau) d\tau,
\end{eqnarray}
where $S(t)w_0$ is the solution of the corresponding linear problem \eqref{pde_tl}. Using estimates of Appendix \ref{apriori} and boundedness of $T_N^{-1}$ on $L^2(0,L)$, we obtain
\begin{equation*}\label{nonlin_loc1}
	\begin{split}
		\|\Psi z\|_{X_T^\ell} &\leq \left\| S(t)w_0 + \kappa \int_0^t S(t - \tau) Fz(\tau) d\tau\right\|_{X_T^\ell} \\
		&\leq c_T\|w_0\|_{H^{\ell}(0,L)} + c_{T,k}\int_0^T \left\|(T_Nz)^3(t) \right\| dt.
	\end{split}
\end{equation*}
Applying the Gagliardo-Nirenberg inequality and using boundedness of the backstepping transformation, we get
\begin{equation} \label{nonlin_loc1.5}
	\begin{split}
		\left\|\left(T_Nz\right)^3(t)\right\| &= \|T_Nz\|_{L^6(0,L)}^3 \\
		&\leq c_L\|\partial_x((T_Nz)(t)\| \|T_Nz(t)\|^2\\
		&\leq c_{k,L} \|z_x(t)\| \|z(t)\|^2,
	\end{split}
\end{equation}
from which it follows that
\begin{equation} \label{nonlin_loc2}
	\begin{split}
	\int_0^T \left\|((I + \Upsilon_kP_N)z)^3(t) \right\| dt &\leq \int_0^T \left(c_{k,L}\|z_x(t)\| \|z(t)\|^2 \right) dt \\
	&\leq c_{k,L} \sqrt{T}\|z\|_{L^\infty(0,T;L^2(0,L))}^2 \|z_x\|_{L^2(0,T;L^2(0,L))} \\
	&\leq c_{k,L}\sqrt{T}\|z\|_{X_T^\ell}^3.
	\end{split}
\end{equation}
Combining \eqref{nonlin_loc1.5}-\eqref{nonlin_loc2}, we get $$\label{fixed1}\|\Psi z\|_{X_T^\ell} \leq c_T\left(\|w_0\|_{H^{\ell}(0,L)} + \sqrt{T}\|z\|_{X_T^\ell}^3\right),$$ where $c_T$ is uniformly bounded in $T$, and it also depends on parameters $\nu,\alpha,\mu, k, L$, which are fixed throughout.

To prove local existence, it is enough to show that the mapping \eqref{wp_map2} has a fixed point for some $T > 0$. Let $B_{R}^T \doteq \left\{ z \in X_T^\ell \, | \, \|z\|_{X_T^\ell} \leq R \right\}$, where $R \doteq 2c_{T} \|w_0\|_{H^{\ell}(0,L)}$. Then,
$\|\Psi z\|_{X_T^\ell} \leq \frac{R}{2} + c_{T}\sqrt{T} R^3.$ Let us choose $T > 0$ small enough that $2 c_{T}\sqrt{T}R^2 < 1$ holds. This yields $\|\Psi z\|_{X_T^\ell} < R$, i.e., $\Psi$ maps $B_{R}^T$ onto itself for sufficiently small $T > 0$.

Next, we show that $\Psi$ is a contraction on $B_R^T$ for sufficiently small $T$. Let $z_1, z_2 \in B_R^T$. Then
\begin{equation*} \label{cont1}
	\begin{split}
	\left\|\Psi z_1 - \Psi z_2\right\|_{X_T^\ell} &= \left\| \int_0^t S(\cdot - s) \left(Fz_1 - Fz_2\right) ds \right\|_{X_T^\ell} \\
	&\leq c_{T}\int_0^T \left\|Fz_1 - Fz_2\right\| dt \\
	&= c_{T} \int_0^T \left\| T_N^{-1} \left[\left(T_Nz_1\right)^3 \right.\right. \left.\left.- \left(T_Nz_2\right)^3 \right] \right\| dt \\
	&\leq c_{T} \int_0^T \left\|\left(T_Nz_1\right)^3 - \left(T_Nz_2\right)^3 \right\| dt \\
	&\leq c_{T} \int_0^T \left\|\left(T_N(z_1 - z_2)\right) \left(\left(T_Nz_1\right)^2 \right.\right.
	\left.\left.+ \left(T_Nz_2\right)^2 \right) \right\| dt \\
	&\leq c_{T}\int_0^T \left(\|T_Bz_1\|_{\infty}^2 \right.\left.+\|T_Nz_2\|_{\infty}^2\right)\|T_N(z_1 - z_2)\|dt\\
	&\leq c_{T}\int_0^T \left(\|\partial_x z_1\|\|z_1\| +\|\partial_x z_2\|\|z_2\|\right)\|z_1 - z_2\|dt\\
	&\leq c_{T}\sqrt{T}(\|z_1\|_{X_T^\ell}^2+\|z_2\|_{X_T^\ell}^2)\|z_1 - z_2\|_{X_T^\ell},
	\end{split}
\end{equation*}
where we used the boundedness of the backstepping transformation as well as its inverse.   For sufficiently small $T > 0$, we guarantee that $c_{T} T^\frac{1}{2} R^2 < \frac{1}{2}$. Hence $\Psi$ becomes a contraction on a closed ball $B_R^T$. This yields a unique solution in $B_R^T\subset X_T^\ell$.

\bibliographystyle{amsplain}
\bibliography{ref}

\providecommand{\bysame}{\leavevmode\hbox to3em{\hrulefill}\thinspace}
\providecommand{\MR}{\relax\ifhmode\unskip\space\fi MR }
\providecommand{\MRhref}[2]{%
  \href{http://www.ams.org/mathscinet-getitem?mr=#1}{#2}
}
\providecommand{\href}[2]{#2}
\begin{thebibliography}{10}

\bibitem{AaSmKr}
Ole~Morten Aamo, Andrey Smyshlyaev, and Miroslav Krsti\'{c}, \emph{Boundary
  control of the linearized {G}inzburg-{L}andau model of vortex shedding}, SIAM
  J. Control Optim. \textbf{43} (2005), no.~6, 1953--1971. \MR{2177789}

\bibitem{Aur19}
Jean Auriol, Kirsten~A. Morris, and Florent Di~Meglio, \emph{Late-lumping
  backstepping control of partial differential equations}, Automatica J. IFAC
  \textbf{100} (2019), 247--259. \MR{3883232}

\bibitem{AzTi}
A.~Azouani and E.~S. Titi, \emph{Feedback control of nonlinear dissipative
  systems by finite determining parameters---a reaction-diffusion paradigm},
  Evol. Equ. Control Theory \textbf{3} (2014), no.~4, 579--594. \MR{3274649}

\bibitem{BaVi}
A.~V. Babin and M.~I. Vishik, \emph{Attractors of evolution equations}, Studies
  in Mathematics and its Applications, vol.~25, North-Holland Publishing Co.,
  Amsterdam, 1992, Translated and revised from the 1989 Russian original by
  Babin. \MR{1156492}

\bibitem{Bar1}
V.~Barbu, \emph{Boundary stabilization of equilibrium solutions to parabolic
  equations}, IEEE Trans. Automat. Control \textbf{58} (2013), no.~9,
  2416--2420. \MR{3247159}

\bibitem{BaTi}
V.~Barbu and R.~Triggiani, \emph{Internal stabilization of {N}avier-{S}tokes
  equations with finite-dimensional controllers}, Indiana Univ. Math. J.
  \textbf{53} (2004), no.~5, 1443--1494. \MR{2104285}

\bibitem{BaOzYi}
A.~Batal, T{\"u}rker {\"O}zsar{\i}, and K.~C. Y{\i}lmaz, \emph{Stabilization of
  higher order schr{\"o}dinger equations on a finite interval: Part i},
  Evolution Equations \& Control Theory (2020).

\bibitem{Dejan}
Dejan~M. Bošković and Miroslav Krstić, \emph{Backstepping control of
  chemical tubular reactors}, Computers \& Chemical Engineering \textbf{26}
  (2002), no.~7, 1077--1085.

\bibitem{CeCo}
Eduardo Cerpa and Jean-Michel Coron, \emph{Rapid stabilization for a
  {K}orteweg-de {V}ries equation from the left {D}irichlet boundary condition},
  IEEE Trans. Automat. Control \textbf{58} (2013), no.~7, 1688--1695.
  \MR{3072853}

\bibitem{chafee}
N.~Chafee and E.~F. Infante, \emph{A bifurcation problem for a nonlinear
  partial differential equation of parabolic type}, Applicable Anal. \textbf{4}
  (1974/75), 17--37. \MR{440205}

\bibitem{Che}
A.~Yu. Chebotarev, \emph{Finite-dimensional controllability of systems of
  {N}avier-{S}tokes type}, Differ. Uravn. \textbf{46} (2010), no.~10,
  1495--1503. \MR{2798716}

\bibitem{CoLu}
Jean-Michel Coron and Qi~L\"{u}, \emph{Local rapid stabilization for a
  {K}orteweg-de {V}ries equation with a {N}eumann boundary control on the
  right}, J. Math. Pures Appl. (9) \textbf{102} (2014), no.~6, 1080--1120.
  \MR{3277436}

\bibitem{Feng1}
Hongyinping Feng, Xiao hui Wu, and Bao‐Zhu Guo, \emph{Actuator dynamics
  compensation in stabilization of abstract linear systems}, ArXiv
  \textbf{abs/2008.11333} (2020).

\bibitem{Feng2}
Hongyinping Feng, Pei-Hua Lang, and Jiankang Liu, \emph{Boundary stabilization
  and observation of a weak unstable heat equation in a general
  multi-dimensional domain}, Automatica J. IFAC \textbf{138} (2022), Paper No.
  110152. \MR{4370482}

\bibitem{FoPr}
C.~Foia\c{s} and G.~Prodi, \emph{Sur le comportement global des solutions
  non-stationnaires des \'{e}quations de {N}avier-{S}tokes en dimension {$2$}},
  Rend. Sem. Mat. Univ. Padova \textbf{39} (1967), 1--34. \MR{223716}

\bibitem{FTR}
C.~Foias, O.~Manley, R.~Rosa, and R.~Temam, \emph{Navier-{S}tokes equations and
  turbulence}, Encyclopedia of Mathematics and its Applications, vol.~83,
  Cambridge University Press, Cambridge, 2001. \MR{1855030}

\bibitem{Fur1}
A.~V. Fursikov, \emph{Stabilizability of a quasilinear parabolic equation by
  means of boundary feedback control}, Mat. Sb. \textbf{192} (2001), no.~4,
  115--160. \MR{1834095}

\bibitem{Hale}
Jack~K. Hale, \emph{Asymptotic behavior of dissipative systems}, Mathematical
  Surveys and Monographs, vol.~25, American Mathematical Society, Providence,
  RI, 1988. \MR{941371}

\bibitem{henry}
D.~Henry, \emph{Geometric theory of semilinear parabolic equations}, Lecture
  Notes in Mathematics, vol. 840, Springer-Verlag, Berlin-New York, 1981.
  \MR{610244}

\bibitem{KaTi}
V.~K. Kalantarov and E.~S. Titi, \emph{Global stabilization of the
  {N}avier-{S}tokes-{V}oight and the damped nonlinear wave equations by finite
  number of feedback controllers}, Discrete Contin. Dyn. Syst. Ser. B
  \textbf{23} (2018), no.~3, 1325--1345. \MR{3810120}

\bibitem{KaOz}
J.~Kalantarova and T.~\"{O}zsar{\i}, \emph{Finite-parameter feedback control
  for stabilizing the complex {G}inzburg-{L}andau equation}, Systems Control
  Lett. \textbf{106} (2017), 40--46. \MR{3681011}

\bibitem{KaFr}
Rami Katz and Emilia Fridman, \emph{Finite-dimensional control of the heat
  equation: {D}irichlet actuation and point measurement}, Eur. J. Control
  \textbf{62} (2021), 158--164. \MR{4334663}

\bibitem{Kat23}
\bysame, \emph{Global stabilization of a 1{D} semilinear heat equation via
  modal decomposition and direct {L}yapunov approach}, Automatica J. IFAC
  \textbf{149} (2023), Paper No. 110809, 10. \MR{4527320}

\bibitem{Koga2020}
Shumon Koga and Miroslav Krstic, \emph{Single-boundary control of the two-phase
  stefan system}, Systems and Control Letters \textbf{135} (2020).

\bibitem{krsticbook}
M.~Krstic and A.~Smyshlyaev, \emph{Boundary control of {PDE}s}, Advances in
  Design and Control, vol.~16, Society for Industrial and Applied Mathematics
  (SIAM), Philadelphia, PA, 2008, A course on backstepping designs.
  \MR{2412038}

\bibitem{KrGuSm}
Miroslav Krstic, Bao-Zhu Guo, and Andrey Smyshlyaev, \emph{Boundary controllers
  and observers for schr{\"o}dinger equation}, 2007 46th IEEE Conference on
  Decision and Control, IEEE, 2007, pp.~4149--4154.

\bibitem{Lad1}
O.~Ladyzhenskaya, \emph{Attractors for semigroups and evolution equations},
  Lezioni Lincee. [Lincei Lectures], Cambridge University Press, Cambridge,
  1991. \MR{1133627}

\bibitem{LhPr}
Hugo Lhachemi and Christophe Prieur, \emph{Finite-dimensional observer-based
  boundary stabilization of reaction-diffusion equations with either a
  {D}irichlet or {N}eumann boundary measurement}, Automatica J. IFAC
  \textbf{135} (2022), Paper No. 109955, 9. \MR{4324193}

\bibitem{LHM}
H.~Liu, P.~Hu, and I.~Munteanu, \emph{Boundary feedback stabilization of
  {F}isher's equation}, Systems Control Lett. \textbf{97} (2016), 55--60.
  \MR{3571054}

\bibitem{Liu03}
W.~Liu, \emph{Boundary feedback stabilization of an unstable heat equation},
  SIAM J. Control Optim. \textbf{42} (2003), no.~3, 1033--1043. \MR{2002146}

\bibitem{Moore}
Arden~L. Moore and Li~Shi, \emph{Emerging challenges and materials for thermal
  management of electronics}, Materials Today \textbf{17} (2014), no.~4, 163
  – 174.

\bibitem{Mun1}
I.~Munteanu, \emph{Stabilisation of parabolic semilinear equations}, Internat.
  J. Control \textbf{90} (2017), no.~5, 1063--1076. \MR{3630490}

\bibitem{Mun2}
\bysame, \emph{Boundary stabilization of parabolic equations}, Progress in
  Nonlinear Differential Equations and their Applications, vol.~93,
  Birkh\"{a}user/Springer, Cham, 2019, Subseries in Control. \MR{3890044}

\bibitem{OzAr}
T\"{u}rker \"{O}zsar{\i} and Eda Arabac\i, \emph{Boosting the decay of
  solutions of the linearised {K}orteweg--de {V}ries--{B}urgers equation to a
  predetermined rate from the boundary}, Internat. J. Control \textbf{92}
  (2019), no.~8, 1753--1763. \MR{3976014}

\bibitem{OzBa}
T\"{u}rker \"{O}zsar{\i} and Ahmet Batal, \emph{Pseudo-backstepping and its
  application to the control of {K}orteweg--de {V}ries equation from the right
  endpoint on a finite domain}, SIAM J. Control Optim. \textbf{57} (2019),
  no.~2, 1255--1283. \MR{3934103}

\bibitem{OzYi}
T{\"u}rker {\"O}zsar{\i} and K.~C. Y{\i}lmaz, \emph{Stabilization of higher
  order schr{\"o}dinger equations on a finite interval: Part ii}, Evolution
  Equations \& Control Theory (2021).

\bibitem{Rob}
James~C. Robinson, \emph{Infinite-dimensional dynamical systems}, Cambridge
  Texts in Applied Mathematics, Cambridge University Press, Cambridge, 2001, An
  introduction to dissipative parabolic PDEs and the theory of global
  attractors. \MR{1881888}

\bibitem{sattinger}
D.~H. Sattinger, \emph{Monotone methods in nonlinear elliptic and parabolic
  boundary value problems}, Indiana Univ. Math. J. \textbf{21} (1971/72),
  979--1000. \MR{299921}

\bibitem{SmCeKr}
Andrey Smyshlyaev, Eduardo Cerpa, and Miroslav Krstic, \emph{Boundary
  stabilization of a 1-{D} wave equation with in-domain antidamping}, SIAM J.
  Control Optim. \textbf{48} (2010), no.~6, 4014--4031. \MR{2645471}

\bibitem{SmKr1}
Andrey Smyshlyaev and Miroslav Krstic, \emph{Closed-form boundary state
  feedbacks for a class of 1-{D} partial integro-differential equations}, IEEE
  Trans. Automat. Control \textbf{49} (2004), no.~12, 2185--2202. \MR{2106749}

\bibitem{SmKr2}
\bysame, \emph{Boundary control of an anti-stable wave equation with
  anti-damping on the uncontrolled boundary}, Systems Control Lett. \textbf{58}
  (2009), no.~8, 617--623. \MR{2542119}

\bibitem{Temam}
Roger Temam, \emph{Infinite-dimensional dynamical systems in mechanics and
  physics}, second ed., Applied Mathematical Sciences, vol.~68,
  Springer-Verlag, New York, 1997. \MR{1441312}

\bibitem{Wang2024}
Yimeng Wang, Changxuan Zhou, Qitao Zhao, Ruihan Jia, and Wei Wu,
  \emph{Ventilation control of road tunnels towards disturbance suppression},
  Scientific Reports \textbf{14} (2024), no.~1.

\bibitem{Woit17}
Frank Woittennek, Marcus Riesmeier, and Stefan Ecklebe, \emph{On approximation
  and implementation of transformation based feedback laws for distributed
  parameter systems}, IFAC-PapersOnLine \textbf{50} (2017), no.~1, 6786--6792,
  20th IFAC World Congress.

\bibitem{Yu14}
Xin Yu, Chao Xu, and Jian Chu, \emph{Local exponential stabilization of
  {F}isher's equation using the backstepping technique}, Systems Control Lett.
  \textbf{74} (2014), 1--7. \MR{3278751}

\bibitem{zheng2004nonlinear}
Songmu Zheng, \emph{Nonlinear evolution equations}, Chapman and Hall/CRC, 2004.

\end{thebibliography}
\end{document}